\documentclass[a4paper,english,oneside]{article}
\usepackage[T1]{fontenc}
\usepackage[utf8]{inputenc}

\usepackage{geometry}
\usepackage{amsmath}
\usepackage{amsfonts}
\usepackage{amssymb}
\usepackage{amsthm}
\usepackage[english]{babel}
\usepackage{color}
\usepackage{lmodern}
\usepackage{pstricks-add}
\usepackage{dsfont}
\usepackage{hyperref}
\hypersetup{hidelinks}

\usepackage{mathscinet}
\usepackage{enumitem}
\usepackage[cal=boondoxo,bb=ams]{mathalfa}
\usepackage{microtype}

\usepackage{mathtools}
\usepackage{esint}

\theoremstyle{plain}
\newtheorem{theorem}{Theorem}[section]
\newtheorem{lemma}[theorem]{Lemma}

\newtheorem{proposition}[theorem]{Proposition}

\theoremstyle{definition}
\newtheorem{definition}[theorem]{Definition}

\theoremstyle{remark}
\newtheorem{remark}{Remark}

    \DeclareMathOperator\supp{supp}

    \DeclareMathOperator\Fuj{Fuj}
    
    \DeclareMathOperator\Span{span}
    \DeclareMathOperator\loc{loc}
     \DeclareMathOperator\hor{H}

 \date{}

\begin{document}

\title{Lifespan estimates for local in time solutions to the semilinear heat equation on the Heisenberg group}

\author{Vladimir Georgiev$^\mathrm{a,b,c}$, Alessandro Palmieri$^\mathrm{a}$}

\date{\small{$^\mathrm{a}$ Department of Mathematics, University of Pisa, Largo B. Pontecorvo 5, 56127 Pisa, Italy} \\ $^\mathrm{b}$ Faculty of Science and Engineering, Waseda University 3-4-1, Okubo, Shinjuku-ku, Tokyo 169-8555, Japan \\ $^\mathrm{c}$
Institute of Mathematics and Informatics–BAS
Acad. G. Bonchev Str., Block 8, 1113 Sofia, Bulgaria\\ [2ex] \normalsize{\today} }
\maketitle

\begin{abstract}
In this paper we consider the semilinear Cauchy problem for the heat equation with power nonlinearity in the Heisenberg group $\mathbf{H}_n$. The heat operator is given in this case by $\partial_t-\Delta_{\hor}$, where $\Delta_{\hor}$ is the so-called sub-Laplacian on $\mathbf{H}_n$. We prove that the Fujita exponent $1+2/Q$ is critical, where $Q=2n+2$ is the homogeneous dimension of $\mathbf{H}_n$. Furthermore, we prove sharp lifespan estimates for local in time solutions in the subcritical case and in the critical case. In order to get the upper bound estimate for the lifespan (especially, in the critical case) we employ a revisited test function method developed recently by Ikeda-Sobajima. On the other hand, to find the lower bound estimate for the lifespan we prove a local in time result in  weighted $L^\infty$ space.
\end{abstract}

\begin{flushleft}
\end{flushleft}

\begin{flushleft}
\textbf{AMS Classification (2010)} Primary: 35A01, 35B44, 35R03 ; Secondary: 35K05, 35K08, 35K58

\end{flushleft}

\section{Introduction}

 {The semilinear heat equation  on the Heisenberg group has a critical exponent of Fujita-type. This result is established recently in \cite{{Ruz18}} and the global existence result in the supercritical case is obtained assuming very fast exponential decay of the initial data for the corresponding Cauchy problem.  Our main goal in this work is to derive sharp upper and lower bound estimates for the lifespan of the solution in the subcritical and critical case. Moreover, our goal is to treat the supercritical case and show global existence result using larger space of initial data with polynomial decay at infinity.  }

The Heisenberg group is the Lie group $\mathbf{H}_n=\mathbb{R}^{2n+1}$ equipped with the multiplication rule
\begin{equation*}
(x,y,\tau)\circ (x',y',\tau')= (x+x',y+y',\tau+\tau'+2 (x\cdot y'-x'\cdot y)),
\end{equation*} where $\cdot$ denotes the standard scalar product in $\mathbb{R}^n$. The identity element for $\mathbf{H}_n$ is $0$ and $\eta^{-1}=-\eta$ for any $\eta\in \mathbf{H}_n$.

 A system of left-invariant vector fields that span the Lie algebra $\mathfrak{h}_n$ is given by
\begin{align*}
\partial_\tau, \ X_j  \doteq  \partial_{x_j}+2y_j \,\partial_\tau,   \ Y_j  \doteq  \partial_{y_j}-2x_j  \, \partial_\tau ,
\end{align*} where $1\leq j\leq n$. This system satisfies the commutation relations
\begin{align*}
[X_j,Y_k]=-4\delta_{jk} \, \partial_\tau \quad \mbox{for}  \ 1\leq j,k \leq n.
\end{align*} Therefore, $\mathfrak{h}_n$ is nilpotent and admits the stratification $\mathfrak{h}_n= V_1\oplus V_2$, where $V_1\doteq \Span\{X_j,Y_j\}_{1\leq j\leq n}$ and $V_2\doteq \Span\{\partial_\tau\}$. In other words, $\mathbf{H}_n$ is a 2 step stratified Lie group, whose homogeneous dimension is $Q=2n+2$.
The sub-Laplacian (also known as horizontal Laplacian) on $\mathbf{H}_n$ is defined as
\begin{align}
\Delta_{\hor} & \doteq \sum_{j=1}^n (X^2_j+Y_j^2)
= \Delta_{(x,y)}+4 |(x,y)|^2\partial_{\tau}^2   +4\sum_{j=1}^n\Big(y_j \,\partial_{ x_j \tau}^2-x_j \, \partial_{y_j\tau}^2\Big), \label{def sub laplacian}
\end{align} where $\Delta_{(x,y)}$ and $|(x,y)|$ denote the Laplace operator and the Euclidean norm of $(x,y)$ in $\mathbb{R}^{2n}$, respectively.

Moreover, it is possible to define a metric on $\mathbf{H}_n$. If we denote by $$|(x,y,\tau)|_{\mathbf{H}_n}\doteq \left(\left(|x|^2+|y|^2\right)^2+|\tau|^2\right)^{\frac{1}{4}}$$ the gauge function, then,
\begin{align*}
d(\eta,\zeta)=|\zeta^{-1}\circ\eta |_{\mathbf{H}_n}
\end{align*} is a left-invariant distance on $\mathbf{H}_n$. The gauge $|\cdot |_{\mathbf{H}_n}$ is homogeneous of degree 1 with respect to the family of group automorphisms $\{\delta_r\}_{r>0}$, where $\delta_r$ is the anisotropic dilation $$\delta_r(x,y,\tau)\doteq (rx,ry,r^2 \tau).$$ In particular, in our setting the gauge function satisfies the triangular inequality
\begin{align}\label{triangular ineq gauge}
|\eta \circ \zeta |_{\mathbf{H}_n}\leq |\eta |_{\mathbf{H}_n} +|\zeta |_{\mathbf{H}_n}.
\end{align}

In this paper, we deal with the semilinear Cauchy problem
\begin{align}\label{semilinear heat Heisenberg}
\begin{cases}
u_t-\Delta_{\hor} u= |u|^p, & \eta\in \mathbf{H}_n, \ t>0, \\
u(0,\eta)=\varepsilon u_0(\eta), & \eta\in \mathbf{H}_n,
\end{cases}
\end{align} where $p>1$ and $\varepsilon>0$ is a parameter describing the smallness of the data.

In the Euclidean case, namely, for the Cauchy problem
\begin{align}\label{semilinear heat Eucl}
\begin{cases}
u_t-\Delta u= |u|^p, & x\in \mathbb{R}^n, \ t>0, \\
u(0,x)=\varepsilon u_0(x), & x\in \mathbb{R}^n,
\end{cases}
\end{align}  it is well-known that the critical exponent is the Fujita exponent $$p_{\Fuj}(n)\doteq 1+\frac{2}{n}.$$ In the pioneering paper \cite{Fuj66} Fujita proved a global existence result for $p>p_{\Fuj}(n)$ and the nonexistence of global in time solutions under certain assumptions on the initial data for $1<p< p_{\Fuj}(n)$. Then, Hayakawa \cite{Hay73}, Sugitani \cite{Sug75} and Kobayashi-Sirao-Tanaka \cite{KST77} showed that in the critical case $p=p_{\Fuj}(n)$ it holds a blow-up result as well. In the work \cite{LN92} Lee-Ni determined, among other things, the sharp lifespan estimate of the lifespan for suitably decaying data. More precisely, they showed that the lifespan of local in time solutions to \eqref{semilinear heat Eucl} in the subcritical case behaves as follows
\begin{align*}
T_\varepsilon \simeq \begin{cases} C\varepsilon^{-(\frac{1}{p-1}-\frac{n}{2})^{-1}} & \mbox{if} \ \ 1<p<p_{\Fuj}(n), \\ \exp\big(C \varepsilon ^{-(p-1)}\big) & \mbox{if} \ \ p=p_{\Fuj}(n) \end{cases}.
\end{align*}

Our porpose is to show that $p_{\Fuj}(Q)= 1+\frac{2}{Q}$ is the critical exponent for \eqref{semilinear heat Heisenberg}. Therefore, we will prove both a blow-up result for \eqref{semilinear heat Heisenberg} in the subcritical case $1<p\leq p_{\Fuj}(Q)$ by using the so-called \emph{test function method} (cf. Mitidieri-Pohozaev \cite{MP01}, for example) and a global (in time) existence result for small data solutions in a suitable class of weighted $L^\infty(\mathbf{H}_n)$ spaces in the supercritical case $p>  p_{\Fuj}(Q)$. We point out that really recently Ruzhansky-Yessirkegenov found out in the more general frame of unimodular Lie groups with polynomial volume growth a critical exponent of Fujita-type for the Cauchy problem related to a semilinear  heat equation (where the degenerate sub-Laplacian appears instead of the classical Laplace operator for the Euclidian case in the definition of the heat operator). Nonetheless, their approach, which relies strongly on the semigroup property of the heat semigroup, differs from ours. Indeed, we obtain a global in time result in a different function space. Additionally, we derive the sharp lifespan estimates for local solutions in the subcritical case and in the critical case as well. In particular, for the upper bound estimate in the critical case we employ a technique which has been developed recently by Ikeda-Sobajima and Ikeda-Sobajima-Wakasa in \cite{IS18,IS17,ISW18}. For the lower bound estimate of the lifespan we prove a local in time existence result  in a weighted $L^\infty(\mathbf{H}_n)$ space, slightly modifying the approach for the global existence result of small data solutions in the supercritical case. We point out that the approach with weighted $L^\infty$ spaces is inspired by the tools used in the treatment of the Euclidean (and homogeneous) case in \cite{FGO18} by Fujiwara-Georgiev-Ozawa. In the next section we collect the main results of this paper.


\paragraph{Notations}

 Throughout this paper we will use the following notations: $B^k(R)$ denotes the ball in $\mathbb{R}^k$ around the origin with radius $R$; $f \lesssim g$ means that there exists a positive constant $C$ such that $f \leqslant Cg$ and, similarly, for $f\gtrsim g$; moreover, $f\approx g$ means $f\lesssim g$ and $f\gtrsim g$. Finally, we will consider the Lebesgue measure on $\mathbb{R}^{2n+1}$ (denoted by $d\eta$) as left-invariant Haar measure on $\mathbf{H}_n$.

\section{Main results} \label{Section main results}

In this section we state the main results that we are going to prove in the next sections.

We begin by introducing a suitable notion of weak solution for \eqref{semilinear heat Heisenberg}.

\begin{definition} A \emph{weak solution} of the Cauchy problem \eqref{semilinear heat Heisenberg} in $[0,T)\times \mathbf{H}_n$ is a function $u\in L^p_{\loc}([0,T)\times \mathbf{H}_n)$ that satisfies
\begin{align}
& \int_0^T \int_{\mathbf{H}_n}|u(t,\eta)|^p\varphi(t,\eta) \, d\eta \, dt+ \varepsilon \int_{\mathbf{H}_n}u_0(\eta)\varphi(0,\eta) \, d\eta  = -\int_0^T \int_{\mathbf{H}_n}u(t,\eta)(\partial_t +\Delta_{\hor}) \varphi(t,\eta) \, d\eta \, dt \label{def weak sol formula}
\end{align}
for any $\varphi \in \mathcal{C}_0^\infty([0,T)\times \mathbf{H}_n)$. If $T=\infty$, we call $u$ a \emph{global} in time weak solution to \eqref{semilinear heat Heisenberg}, else we call $u$ a \emph{local} in time weak solution.
\end{definition}

In the next result we provide an upper bound for the lifespan of a local in time solution $u$, which defined as follows
\begin{align*}
T(\varepsilon)\doteq \sup_{ T>0} \big\{u \ \mbox{is a weak solution to \eqref{semilinear heat Heisenberg} in} \ [0,T)\times \mathbf{H}_n\big\}.
\end{align*}

\begin{theorem}\label{Thm upper bound estimate lifespan} Let $1<p\leq p_{\Fuj}(Q)$. We assume that $u_0\in L^1(\mathbf{H}_n)$ satisfies
\begin{equation} \int_{\mathbf{H}_n} u_0(\eta) \, d\eta >0, \label{assumption intial data}
\end{equation}
and is compactly supported with  $\supp u_0 \subset  \{(x,y,\tau)\in \mathbf{H}_n: |x|^2+|y|^2+|\tau|<R_0\}$ for some $R_0>0$ . Then, there exists $\varepsilon_0>0$ such that for any $\varepsilon\in (0,\varepsilon_0]$ it holds
\begin{align}\label{upper bound lifespan}
T(\varepsilon) \leq
 \begin{cases} C \varepsilon^{-\left(\frac{1}{p-1}-\frac{Q}{2}\right)^{-1}} & \mbox{if} \ \ p\in (1,p_{\Fuj}(Q)), \\
 \exp\big(C \varepsilon^{-(p-1)} \big) & \mbox{if} \ \ p=p_{\Fuj}(Q),
\end{cases}
\end{align} where $C$ is independent of $\varepsilon$ and positive constant.
\end{theorem}

We introduce now the definition of the weighted $L^\infty$ spaces, where we will study the existence and the uniqueness results for the Cauchy problem \eqref{semilinear heat Heisenberg}. Let $\kappa>0$ be a parameter. Then, we define
\begin{align*}
\mathrm{X}_{\kappa,T} & \doteq   \big(1+t+|\eta|^2_{\mathbf{H}_n}\big)^{-\frac{\kappa}{2}} L^\infty\big(0,T\, ;L^\infty(\mathbf{H}_n) \big),  \\
\mathrm{X}_{\kappa} & \doteq  \big(1+t+|\eta|^2_{\mathbf{H}_n}\big)^{-\frac{\kappa}{2}} L^\infty\big(0,\infty\, ;L^\infty(\mathbf{H}_n) \big),
\end{align*}  equipped with the norms
\begin{align*}
\| u\|_{\mathrm{X}_{\kappa,T}} & \doteq \big\| \big(1+t+|\eta|^2_{\mathbf{H}_n}\big)^{\frac{\kappa}{2}} u(t,\eta)\big\|_{L^\infty
([0,T)\times \mathbf{H}_n)} , \\
\| u\|_{\mathrm{X}_\kappa} & \doteq \big\| \big(1+t+|\eta|^2_{\mathbf{H}_n}\big)^{\frac{\kappa}{2}} u(t,\eta)\big\|_{L^\infty
([0,\infty)\times \mathbf{H}_n)},
\end{align*} respectively. For the existence results (either global or local in time) we will consider mild solutions to \eqref{semilinear heat Heisenberg}. Therefore, let us recall the definition of mild solution in the next definition.

\begin{definition}\label{Def mild sol} Let $\kappa$ be a positive real number. A \emph{mild solution} of the Cauchy problem \eqref{semilinear heat Heisenberg} in $\mathrm{X}_{\kappa,T}$ is a function $u\in \mathrm{X}_{\kappa,T}$ that satisfies the nonlinear integral equation
\begin{align} \label{nonlinear integral equation}
u(t)= \varepsilon \, e^{t\Delta_{\hor}} u_0+\int_0^t e^{(t-s)\Delta_{\hor}}|u(s)|^p \, ds
\end{align} for any $t\in [0,T)$.  If $T=\infty$, we call $u$ a \emph{global} in time mild solution to \eqref{semilinear heat Heisenberg}, else we call $u$ a \emph{local} in time mild solution.
\end{definition}

Finally, we may state the global in time existence result for small data solutions in the supercritical case in the family of weighted function spaces $\{\mathrm{X}_{\kappa}\}_{\kappa\in (0,Q)}$ and the local in time existence result in the subcritical and critical case in the weighted space $\mathrm{X}_{Q,T}$.

\begin{theorem} \label{Thm GESD} Let us assume $p>p_{\Fuj}(Q)$. Let us consider $\kappa= \frac{2}{p-1}\in (0,Q)$. Then, there exists $\varepsilon_0=\varepsilon_0(n,p,u_0)>0$ such that for any $\varepsilon\in (0,\varepsilon_0]$ and $u_0\in  (1+|\cdot|_{\mathbf{H}_n}^2)^{-\frac{\kappa}{2}} L^\infty(\mathbf{H}_n)$ there exists a unique global in time mild solution $u$ to \eqref{semilinear heat Heisenberg} in the weighted $L^{\infty}$ space $\mathrm{X}_\kappa$. Furthermore, $u$ satisfies the decay estimate
\begin{align*}
|u(t,\eta)|\lesssim \varepsilon \,  \| \,(1+|\cdot|_{\mathbf{H}_n}^2)^{\frac{\kappa}{2}} u_0\|_{L^\infty(\mathbf{H}_n)} \big(1+t+|\eta|^2_{\mathbf{H}_n}\big)^{-\frac{\kappa}{2}} \qquad \mbox{for any} \ (t,\eta)\in [0,\infty)\times \mathbf{H}_n.
\end{align*}
\end{theorem}

\begin{theorem}\label{Thm LER} Let us assume $1<p\leq p_{\Fuj}(Q)$. Let us consider $\kappa >Q$. Then, for any $u_0\in (1+ |\cdot|_{\mathbf{H}_n}^2)^{-\frac{\kappa}{2}} L^\infty(\mathbf{H}_n)$ there exists a unique local in time mild solution $u$ to \eqref{semilinear heat Heisenberg} on $[0,T_\varepsilon)$ in the weighted $L^{\infty}$ space $\mathrm{X}_{Q,T_\varepsilon}$. Furthermore, the following lower bound estimate for the lifespan of $u$ holds:
\begin{align}\label{lower bound lifespan}
T_\varepsilon\geq \begin{cases}C  \varepsilon^{-(\frac{1}{p-1}-\frac{Q}{2})^{-1}} & \mbox{if} \ \ p<p_{\Fuj}(Q), \\ \exp\big( C  \varepsilon^{-(p-1)} \big) & \mbox{if} \ \ p=p_{\Fuj}(Q), \end{cases}
\end{align} where $C=C(Q,p)$ is a positive and independent of $\varepsilon$ constant.
\end{theorem}

\section{Blow-up results}

\subsection{Test function method}

%


In this subsection, we prove a result which is already known in the literature (for example, see \cite[Theorem 3.1]{PV00}). Nevertheless, since we are going to modify this approach (the \emph{test function method}) in order to derive the upper bound estimate for the lifespan in the subcritical case, for the sake of self-containedness and readability of the paper we include briefly its proof.

\begin{proposition} \label{Prop test function method} Let $1<p\leq p_{\Fuj}(Q)$, where $Q=2n+2$ is the homogeneous dimension of $\mathbf{H}_n$. If we assume that $u_0\in L^1(\mathbf{H}_n)$ satisfies
\begin{equation} \liminf_{R\to \infty}\int_{D_R} u_0(\eta) \, d\eta >0, \label{assumption intial data TFM}
\end{equation} where $D_R\doteq B^n(R)\times B^n(R)\times [-R^2,R^2]$, then, there exists no global in time weak solution to \eqref{semilinear heat Heisenberg}.
\end{proposition}


\begin{proof}
We apply the so-called \emph{test function method}. By contradiction, we assume that there exists a global in time weak solution $u$ to \eqref{semilinear heat Heisenberg}.

 Let us consider two bump functions $\alpha\in \mathcal{C}_0^\infty(\mathbb{R}^n)$ and $\beta\in \mathcal{C}_0^\infty(\mathbb{R})$. Furthermore, we require that $\alpha,\beta$ are radial symmetric and decreasing with respect to the radial variable,  $\alpha=1$ on $B^{n}(\frac{1}{2})$, $\beta=1$ on $[-\frac{1}{4},\frac{1}{4}]$, $\supp\alpha \subset B^{n}(1)$ and $\supp\beta \subset (-1,1)$.
  If $R>0$ is a parameter, then, we define the test function $\varphi_R\in \mathcal{C}^\infty_0([0,\infty)\times \mathbb{R}^{2n+1})$ with separate variables as follows
 \begin{align}
 \varphi_R(t,x,y,\tau) \doteq \beta\left(\tfrac{t}{R^2}\right)\alpha\left(\tfrac{x}{R}\right)\alpha\left(\tfrac{y}{R}\right)\beta\left(\tfrac{\tau}{R^2}\right) \quad \mbox{for any} \ (t,x,y,\tau)\in [0,\infty)\times \mathbb{R}^{2n+1}. \label{def varphiR}
 \end{align} It is well-know that
 \begin{align*}
 |\partial_j \alpha| & \lesssim \alpha^{\frac{1}{p}} \quad \mbox{for any} \ 1\leq j\leq n, \quad
  |\partial_j \partial_k \alpha|  \lesssim \alpha^{\frac{1}{p}} \quad \mbox{for any} \ 1\leq j,k \leq n, \quad
  |\beta '|  \lesssim \beta^{\frac{1}{p}}, \quad |\beta''| \lesssim \beta^{\frac{1}{p}}.
 \end{align*} Furthermore, $0\leq \alpha,\beta \leq 1$ implies immediately $\alpha\leq \alpha^{\frac{1}{p}} $ and $\beta\leq \beta^{\frac{1}{p}} $. Therefore, from the relations
\begin{align*}
 \partial_t \varphi_R(t,x,y,\tau) &= R^{-2} \beta'\left(\tfrac{t}{R^2}\right)\alpha\left(\tfrac{x}{R}\right)\alpha\left(\tfrac{y}{R}\right)\beta\left(\tfrac{\tau}{R^2}\right), \\
 \Delta_{\hor}\varphi_R(t,x,y,\tau) &= R^{-2} \beta\left(\tfrac{t}{R^2}\right)\Delta \alpha\left(\tfrac{x}{R}\right)\alpha\left(\tfrac{y}{R}\right)\beta\left(\tfrac{\tau}{R^2}\right) +R^{-2} \beta\left(\tfrac{t}{R^2}\right) \alpha\left(\tfrac{x}{R}\right)\Delta \alpha\left(\tfrac{y}{R}\right)\beta\left(\tfrac{\tau}{R^2}\right) \\
 & \quad + 4 R^{-3}\sum _{j=1}^n y_j \beta\left(\tfrac{t}{R^2}\right)\partial_j \alpha\left(\tfrac{x}{R}\right)\alpha\left(\tfrac{y}{R}\right)\beta'\left(\tfrac{\tau}{R^2}\right) - 4 R^{-3}\sum _{j=1}^n x_j \beta\left(\tfrac{t}{R^2}\right)\alpha\left(\tfrac{x}{R}\right)\partial_j \alpha\left(\tfrac{y}{R}\right)\beta'\left(\tfrac{\tau}{R^2}\right)  \\
 & \quad + 4 R^{-4} (|x|^2+|y|^2) \beta\left(\tfrac{t}{R^2}\right)\alpha\left(\tfrac{x}{R}\right)\alpha\left(\tfrac{y}{R}\right)\beta''\left(\tfrac{\tau}{R^2}\right),
\end{align*} where $\Delta$ denotes the Laplace operator on $\mathbb{R}^n$, we get
\begin{equation} \label{estimate derivatives of varphi_R}
\begin{split}
 | \partial_t \varphi_R| & \lesssim R^{-2} ( \varphi_R)^{\frac{1}{p}},  \\ | \Delta_{\hor} \varphi_R| & \lesssim R^{-2} (\varphi_R)^{\frac{1}{p}}.
 \end{split}
\end{equation} Note we employed the fact that $\supp \varphi_R \subset [0,R^2]\times B^n(R)\times B^n(R) \times [-R^2,R^2]$ in order to estimate the polynomial terms in the estimate of $|\Delta_{\hor} \varphi_R|$.

Let us apply the definition of weak solution \eqref{def weak sol formula} for the test function $\varphi_R$. Hence, by \eqref{estimate derivatives of varphi_R} we obtain
\begin{align}
 \int_0^\infty \int_{\mathbf{H}_n} & |u(t,\eta)|^p\varphi_R(t,\eta) \, d\eta \, dt+ \varepsilon \int_{\mathbf{H}_n}u_0(\eta)\varphi_R(0,\eta) \, d\eta \notag\\
 & \leq  \int_0^\infty \int_{\mathbf{H}_n}|u(t,\eta)|\big(|\partial_t \varphi_R(t,\eta)| +|\Delta_{\hor} \varphi_R(t,\eta)|\big)  \, d\eta \, dt \notag \\
 &  \lesssim  R^{-2} \int_0^\infty \int_{\mathbf{H}_n}|u(t,\eta)| (\varphi_R(t,\eta))^{\frac{1}{p}} \, d\eta \, dt \notag  \\
 & \leq  R^{-2}  \bigg(\int_0^\infty \int_{\mathbf{H}_n} |u(t,\eta)|^p \varphi_R(t,\eta) \, d\eta \, dt \bigg)^{\frac{1}{p}}\bigg(\iint_{[0,R^2]\times D_R} d\eta \, dt  \bigg)^{\frac{1}{p'}}.\label{chain ineq weak sol varphi_R}
\end{align} Let us introduce now the functions
\begin{align}
I_{R}\doteq \int_0^\infty \int_{\mathbf{H}_n}  |u(t,\eta)|^p\varphi_R(t,\eta) \, d\eta \, dt, \quad
J_{R}\doteq \int_{\mathbf{H}_n}u_0(\eta)\varphi_R(0,\eta) \, d\eta. \label{def IR and JR}
\end{align} Due to the assumption on the data \eqref{assumption intial data TFM}, we have $\liminf_{R\to \infty} J_R>0$, which implies in turn that $J_R>0$ for $R\geq R_0$, where $R_0$ is a suitable positive real number. Indeed, from $\supp \varphi_R(0,\cdot)\subset D_R$ and $\varphi_R(0,\cdot)=1$ on $D_{R/2}$ we get trivially
\begin{align*}
J_R= \int_{D_R}u_0(\eta)\varphi_R(0,\eta) \, d\eta \geq \int_{D_{R/2}}u_0(\eta) \, d\eta.
\end{align*} Then, for $R\geq R_0$ the estimate in \eqref{chain ineq weak sol varphi_R} yields
\begin{align}\label{estimate I_R intermediate}
I_R\leq I_R+\varepsilon J_R \lesssim R^{-2+\frac{2n+4}{p'}} I_R^{\frac{1}{p}}=R^{Q-\frac{Q+2}{p}} I_R^{\frac{1}{p}}.
\end{align} When the exponent of $R$ in the right-hand side of the last inequality  is negative, i.e. for $p<p_{\Fuj}(Q)$, we have that $$0\leq I_R^{1-\frac{1}{p}}\lesssim R^{Q-\frac{Q+2}{p}}\longrightarrow  0 \quad \mbox{as}\  R\to \infty.$$ Thus, $\lim_{R\to \infty} I_R =0$. However, this is not possible, because the term $J_R$ is positive for $R$ sufficiently large. So, letting $R\to \infty$ in \eqref{estimate I_R intermediate} we find the contradiction we were looking for. In order to get a contradiction in the critical case $p=p_{\Fuj}(Q)$ too, we need to refine the estimate in \eqref{chain ineq weak sol varphi_R}. More precisely, we can use the fact that $\partial_t \varphi_R$ is supported in $\widehat{P}_R\doteq [\frac{R^2}{4},R^2]\times D_R$ and $\Delta_{\hor} \varphi_R$  is supported in $\widetilde{P}_R\doteq[0,R^2]\times (D_{1,R}\cup D_{2,R}\cup D_{3,R})$, where
\begin{align*}
D_{1,R} & \doteq \big(B^n(R)\setminus B^n(\tfrac{R}{2})\big)\times B^n(R)\times [-R^2,R^2] ,\\
D_{2,R} & \doteq B^n(R)\times \big(B^n(R)\setminus B^n(\tfrac{R}{2})\big)\times  [-R^2,R^2] ,\\
D_{3,R} & \doteq  B^n(R) \times (B^n(R))\times \big( [-R^2,R^2]\setminus \big[-\tfrac{R^2}{4},\tfrac{R^2}{4}\big]\big).
\end{align*} Consequently, for $R\geq R_0$ we may improve \eqref{chain ineq weak sol varphi_R} as follows
\begin{align}\label{estimate I_R intermediate improvement}
I_R\leq I_R+\varepsilon J_R\lesssim  \widehat{I}_R^{\frac{1}{p}}+\widetilde{I}_R^{\frac{1}{p}},
\end{align} where
\begin{align*}
\widehat{I}_R  \doteq \iint_{\widehat{P}_R}  |u(t,\eta)|^p\varphi_R(t,\eta) \, d\eta \, dt \quad \mbox{and} \quad
\widetilde{I}_R \doteq \iint_{\widetilde{P}_R}  |u(t,\eta)|^p\varphi_R(t,\eta) \, d\eta \, dt.
\end{align*} In the critical case $p=p_{\Fuj}(Q)$, from \eqref{estimate I_R intermediate} it follows that $I_R$ is uniformly bounded as $R\to\infty$. Using the monotone convergence theorem, we find
\begin{align*}
\lim_{R\to\infty} I_R =  \lim_{R\to\infty} \int_0^\infty \int_{\mathbf{H}_n}  |u(t,\eta)|^p\varphi_R(t,\eta) \, d\eta \, dt = \int_0^\infty \int_{\mathbf{H}_n}  |u(t,\eta)|^p  \, d\eta \, dt \lesssim 1.
\end{align*} This means that $u\in L^p([0,\infty)\times \mathbf{H}_n)$. Applying now the dominated convergence theorem, as the characteristic functions of the sets $\widehat{P}_R$ and $\widetilde{P}_R$ converge to the zero function for $R\to \infty$, we have
\begin{align*}
\lim_{R\to\infty} \widehat{I}_R &= \lim_{R\to\infty} \iint_{\widehat{P}_R}  |u(t,\eta)|^p\varphi_R(t,\eta) \, d\eta \, dt =0,\\
\lim_{R\to\infty} \widetilde{I}_R &= \lim_{R\to\infty} \iint_{\widetilde{P}_R}  |u(t,\eta)|^p\varphi_R(t,\eta) \, d\eta \, dt =0.
\end{align*} Also, letting $R\to \infty$, \eqref{estimate I_R intermediate improvement} implies $\lim_{R\to\infty}I_R=0$ which provides the desired contradiction in turn, as we have already seen in the subcritical case. The proof is completed.
\end{proof}

\begin{remark} In Subsection \ref{Section upper bound} we will provide the complete proof of Theorem \ref{Thm upper bound estimate lifespan}. However, in the subcritical case $1<p<p_{\Fuj}(Q)$ it is possible to prove the upper bound estimate for the lifespan of the solution by modifying slightly the approach used in the proof of Proposition \ref{Prop test function method}. In fact, by \eqref{estimate I_R intermediate} we know that
\begin{align*}
I_R+\varepsilon J_R\leq C R^{Q-\frac{Q+2}{p}} I_R^{\frac{1}{p}},
\end{align*} where $I_R$ and $J_R$ are defined as in \eqref{def IR and JR}. Applying Young's inequality on the right-hand side of the previous inequality, we get
\begin{align*}
I_R+\varepsilon J_R\leq \tfrac{1}{p'}\left( C R^{Q-\frac{Q+2}{p}}\right)^{p'}+\tfrac{1}{p} I_R
\end{align*} which implies in turn
\begin{align*}
 \varepsilon J_R \leq (1-\tfrac{1}{p})I_R+\varepsilon J_R \lesssim \left( R^{Q-\frac{Q+2}{p}}\right)^{p'} = R^{-\left(\frac{2}{p-1}-Q\right)}.
\end{align*}
Due to the assumption \eqref{assumption intial data TFM}, we have seen that $\liminf_{R\to \infty} J_R>0$. This means that $J_R\gtrsim 1$ for $R\geq R_1$, where $R_1$ is a suitable large constant. Hence, for $R\geq R_1$ we find
\begin{align*}
\varepsilon \lesssim \varepsilon J_R \lesssim  R^{-\left(\frac{2}{p-1}-Q\right)}.
\end{align*} If we assume that $1<p<p_{\Fuj}(Q)$, then, the power for $R$ is negative in the last estimate. Thus,
\begin{align*}
R\lesssim  \varepsilon^{-\left(\frac{2}{p-1}-Q\right)^{-1}}.
\end{align*}
 We point out that in the scaling of the bump function $\beta$ correspondingly to the time variable 
  in \eqref{def varphiR} the parameter $R^2$ has to be dominated by the lifespan $T$ in order to guarantee $\varphi_R\in \mathcal{C}^\infty_0([0,T)\times\mathbf{H}_n)$. Therefore, the last relation implies
\begin{align*}
T^{\frac{1}{2}}\lesssim \varepsilon^{-\left(\frac{2}{p-1}-Q\right)^{-1}} \ \ \Rightarrow  \ \ T\lesssim \varepsilon^{-\left(\frac{1}{p-1}-\frac{Q}{2}\right)^{-1}},
\end{align*} which is the desired estimate. Note that we assumed  without loss of generality in the previous step that $T^{\frac{1}{2}}\geq R_1$. Indeed, if $T^{\frac{1}{2}}\leq R_1$, then, for $\varepsilon>0$ sufficiently small the inequality $T\lesssim \varepsilon^{-\left(\frac{1}{p-1}-\frac{Q}{2}\right)^{-1}}$ is trivially satisfied.
\end{remark}

\subsection{Upper bound estimates for the lifespan} \label{Section upper bound}

%


In this subsection we prove Theorem \ref{Thm upper bound estimate lifespan}. Our approach is based on the revisited version of the test function method developed by Ikeda-Sobajima in \cite{IS17}. Of course, in the previous subsection we showed how   it is possible to get the upper bound estimate for the lifespan  in the subcritical case, so only the critical case is left. Nonetheless, in the next proof we can deal with the subcritical and critical case at the same time with small modifications and only in the very last steps.

\begin{proof}[Proof of Theorem \ref{Thm upper bound estimate lifespan}] 
Let us begin pointing out that we may assume that $2R_0<T(\varepsilon)$ without loss of generality. Indeed, if $T(\varepsilon)\leq 2R_0$, then, \eqref{upper bound lifespan} is trivially fulfilled, provided that $\varepsilon_0$ is sufficiently small.  Let $\phi\in \mathcal{C}_0^\infty([0,\infty))$ be a bump function such that $\phi=1$ on $[0,\frac{1}{2}]$, $\supp \phi \subset [0,1)$ and $\phi$ is a decreasing function. Furthermore, {we denote
$$\phi^*(r)= \begin{cases} 0 & \mbox{if} \ \  r\in \big[0,\frac{1}{2}\big), \\ \phi(r) & \mbox{if} \ \ r\in \big[\frac{1}{2}, \infty\big). \end{cases}$$    Clearly, $\phi^*$ is not smooth. In some sense, we will use this notation in order to keep trace of the supports of the derivatives of $\phi$, which are strictly contained in the one of $\phi$.}

Let us consider
\begin{align}\label{definition of psi_R}
\psi_R(t,\eta) \doteq \left[\phi\left(s_R(t,\eta)\right)\right]^{2p'}, \quad \psi_R^*(t,\eta) \doteq \left[\phi^*\left(s_R(t,\eta)\right)\right]^{2p'},
\end{align} where $R>0$ is a positive parameter and
\begin{align*}
s_R(t,\eta) \doteq \frac{t^2+|x|^4+|y|^4+|\tau|^2}{R^2}\quad \mbox{for any} \ \eta= (x,y,\tau)\in \mathbf{H}_n.
\end{align*}
 As straightforward consequence of the choice of the function $\phi$
, we get
\begin{align*}
\supp \psi_R &  \subset  \mathcal{Q}_R  \doteq[0,R]\times B^n(R^{\frac{1}{2}}) \times B^n(R^{\frac{1}{2}}) \times [-R,R] .
\end{align*}
Moreover, the relation
\begin{align*}
\partial_t \psi_R(t,\eta) = 4p' R^{-2} t\left[\phi\left(s_R(t,\eta)\right)\right]^{2p'-1} \phi'\left(s_R(t,\eta)\right)
\end{align*}
 implies immediately
\begin{align}
|\partial_t \psi_R(t,\eta)| & \lesssim   R^{-2} t \left[\phi^*\left(s_R(t,\eta)\right)\right]^{\frac{2p'}{p}} \phi\left(s_R(t,\eta)\right) |\phi'\left(s_R(t,\eta)\right)| \lesssim   R^{-1} \left[\psi_R^*(t,\eta)\right]^{\frac{1}{p}}. \label{estimate dt psi_R}
\end{align}
 Similarly, plugging the relations
\begin{align*}
\partial_{x_j}^2 \psi_R(t,\eta) & =8p' R^{-2} (2x_j^2+|x|^2) \left[\phi\left(s_R(t,\eta)\right)\right]^{2p'-1} \phi'\left(s_R(t,\eta)\right)  \\
& \quad +  32p' (2p'-1) R^{-4} |x|^4 x_j^2  \left[\phi\left(s_R(t,\eta)\right)\right]^{2p'-2} \left[\phi'\left(s_R(t,\eta)\right)\right]^2\\
& \quad + 32p' R^{-4} |x|^4 x_j^2 \left[\phi\left(s_R(t,\eta)\right)\right]^{2p'-1} \phi''\left(s_R(t,\eta)\right), \\
\partial_{x_j\tau}^2 \psi_R(t,\eta) & =16p'(2p'-1) R^{-4} |x|^2x_j \tau  \left[\phi\left(s_R(t,\eta)\right)\right]^{2p'-2} \left[\phi'\left(s_R(t,\eta)\right)\right]^2  \\
& \quad + 16p' R^{-4} |x|^2x_j \tau   \left[\phi\left(s_R(t,\eta)\right)\right]^{2p'-1} \phi''\left(s_R(t,\eta)\right), \\
  \partial_{\tau}^2 \psi_R(t,\eta) & =  4p' R^{-2}  \left[\phi\left(s_R(t,\eta)\right)\right]^{2p'-1} \phi'\left(s_R(t,\eta)\right) \\ & \quad +8p'(2p'-1) R^{-4} \tau^2  \left[\phi\left(s_R(t,\eta)\right)\right]^{2p'-2} \left[\phi'\left(s_R(t,\eta)\right)\right]^2  \\
  & \quad +8p' R^{-4}  \left[\phi\left(s_R(t,\eta)\right)\right]^{2p'-1} \phi''\left(s_R(t,\eta)\right)
\end{align*}
 and analogous relations for $\partial_{y_j}^2 \psi_R(t,\eta) $ and $\partial_{y_j\tau}^2 \psi_R(t,\eta)$  in the definition of sub-Laplacian in \eqref{def sub laplacian}, we find the estimate
\begin{align}
|\Delta_{\hor} \psi_R(t,\eta)| & \lesssim    R^{-1} \left[\psi_R^*(t,\eta)\right]^{\frac{1}{p}}.\label{estimate subLapl psi_R}
\end{align} So, applying \eqref{def weak sol formula} with test function $\psi_R$ and using \eqref{estimate dt psi_R}, \eqref{estimate subLapl psi_R}, we obtain
\begin{align*}
\int_0^T \int_{\mathbf{H}_n}|u(t,\eta)|^p\psi_R(t,\eta) \, d\eta \, dt+ \varepsilon \int_{\mathbf{H}_n}u_0(\eta) \, d\eta  & \leq \int_0^T \int_{\mathbf{H}_n}|u(t,\eta)|\big(|\partial_t \psi_R(t,\eta)| +|\Delta_{\hor}\psi_R(t,\eta)| \big)\, d\eta \, dt
\\
& \lesssim R^{-1} \int_0^T \int_{\mathbf{H}_n}|u(t,\eta)|\left[\psi_R^*(t,\eta)\right]^{\frac{1}{p}}\, d\eta \, dt \\
& \lesssim R^{-1+\frac{n+2}{p'}} \bigg(\int_0^T \int_{\mathbf{H}_n}|u(t,\eta)|^p \psi_R^*(t,\eta)\, d\eta \, dt\bigg)^{\frac{1}{p}}
\end{align*} for any $R\in (2R_0,T(\varepsilon))$, where in the last inequality we used H\"{o}lder's inequality and the fact that the Lebesgue measure of $\mathcal{Q}_R$ is  $R^{n+2}$ times a multiplicative constant. Note that the requirement $R>2R_0$ implies that $\psi_R(0,\cdot)\equiv 1$ on $\supp u_0$.

Let us remark that the exponent for $R$ in the right-hand side of the previous chain of inequalities
\begin{align*}
-1+\tfrac{n+2}{p'} & =-1+(n+2)\left(1-\tfrac{1}{p}\right)=n+1-\tfrac{n+2}{p}= \tfrac{1}{p}\left((n+1)(p-1)-1\right) = -\tfrac{p-1}{p}\left(\tfrac{1}{p-1}-(n+1)\right)  \\ &= -\tfrac{p-1}{p}\left(\tfrac{1}{p-1}-\tfrac{Q}{2}\right)
\end{align*} is non-positive if and only if $p\leq p_{\Fuj}(Q)$. Hence, summarizing, we have just shown
\begin{equation}
\int_0^T \int_{\mathbf{H}_n}|u(t,\eta)|^p\psi_R(t,\eta) \, d\eta dt +\varepsilon \int_{\mathbf{H}_n}u_0(\eta) \, d\eta  \lesssim R^{-\frac{p-1}{p}\left(\frac{1}{p-1}-\frac{Q}{2}\right)} \bigg(\int_0^T \int_{\mathbf{H}_n}|u(t,\eta)|^p \psi_R^*(t,\eta)\, d\eta \, dt\bigg)^{\frac{1}{p}} . \label{fundamental estimate lifespan}
\end{equation}

{ Let us use the notations
$$ X(r)  \doteq  \int_0^T \int_{\mathbf{H}_n}|u(t,\eta)|^p\psi_r(t,\eta) \, d\eta dt $$
$$ Y(r) \doteq  \int_0^T \int_{\mathbf{H}_n}|u(t,\eta)|^p \psi_r^*(t,\eta)\, d\eta \, dt$$
for the quantities appearing in the above inequality \eqref{fundamental estimate lifespan}.
We shall need the following simple observation.
\begin{lemma} \label{lemma g}
If $g=g(s)$ is a measurable function satisfying the properties: $g(s)=0$ for $s \in [0,\frac{1}{2}]\cup [1,\infty)$ and $g(s)$ is a decreasing function for $s > 1/2$, then  for any $R >0, A>0$ we have
\begin{equation}\label{estimate lemma g}
 \int_0^R g\left(\frac{A}{r^2}\right) \frac{dr}{r} \leq \frac{\log 2}{2} \, g\left(\frac{A}{R^2}\right).
 \end{equation}
\end{lemma}
\begin{proof}
If the set $\{r: \frac{1}{2}<\frac{A}{r^2}<1\}$ has empty intersection with the domain of integration $[0,R]$, then \eqref{estimate lemma g} is trivially true as the integrand function on the left hand side is identically 0. Otherwise, thanks to the assumptions on $g$ we get immediately
\begin{align*}
\int_0^R g\left(\frac{A}{r^2}\right) \frac{dr}{r} = \int_{[0,R]\cap [\sqrt{A},\sqrt{2A}]} g\left(\frac{A}{r^2}\right) \frac{dr}{r} \leq g\left(\frac{A}{R^2}\right)   \int_{\sqrt{A}}^{\sqrt{2A}}  \frac{dr}{r}  = \frac{\log 2}{2} \, g\left(\frac{A}{R^2}\right).
\end{align*}
\end{proof}
Rewriting \eqref{fundamental estimate lifespan} as 
\begin{equation} \label{si1}
     X(R) + \varepsilon I[u_0] \lesssim R^{-\frac{p-1}{p}\left(\frac{1}{p-1}-\frac{Q}{2}\right)} Y(R)^{\frac{1}{p}},
\end{equation}
where $I[u_0]\doteq \displaystyle{ \int_{\mathbb{R}^{2n+1}}u_0(\eta) \, d\eta}$, and using Lemma \ref{lemma g} with $g = [\phi^*]^{2p'}$ and $A=s_1(t,\eta)$, we easily get
\begin{equation} \label{si2}
 \frac{2}{\log 2} \,  \int_0^R Y(r) \frac{dr}{r} \leq  X(R)  
\end{equation}
Setting $$ W(R)  \doteq  \int_0^R Y(r) \frac{dr}{r} $$ and using $RW^\prime(R) = Y(R),$
we can combine \eqref{si1} and \eqref{si2} and deduce
$$  \frac{2 W(R) }{\log 2} + \varepsilon I[u_0] \leq  X(R)+  \varepsilon I[u_0] \lesssim  R^{-\frac{p-1}{p}\left(\frac{1}{p-1}-\frac{Q}{2}\right)+\frac{1}{p}} \left( W^\prime(R)  \right)^{\frac{1}{p}}.$$ In this way, we arrive at
\begin{align}\label{differential inequality for W(R)}
C R^{\left(\frac{1}{p-1}-\frac{Q}{2}\right)(p-1)-1}\leq W'(R)\left(\varepsilon I[u_0]+\frac{2W(R)}{\log 2}\right)^{-p},
\end{align} where $C$ is suitable positive multiplicative constant  that may change from line to line in the next estimates.
}

The next step is to integrate \eqref{differential inequality for W(R)} over $[2R_0,T(\varepsilon)]$. Clearly,
\begin{align*}
\int_{2R_0}^{T} R^{\left(\frac{1}{p-1}-\frac{Q}{2}\right)(p-1)-1}\, dR \simeq  \begin{cases} T^{\left(\frac{1}{p-1}-\frac{Q}{2}\right)(p-1)}-(2R_0)^{\left(\frac{1}{p-1}-\frac{Q}{2}\right)(p-1)} & \mbox{if} \ \ p\in(1,p_{\Fuj}(Q)), \\ \log\left(\frac{T}{2R_0}\right) & \mbox{if} \ \ p=p_{\Fuj}(Q)\end{cases}
\end{align*} and
\begin{align*}
\int_{2R_0}^{T}  W'(R)\left(\varepsilon I[u_0]+\frac{2W(R)}{\log 2}\right)^{-p} \, dR \leq \frac{\log 2}{2(p-1)} \left(\varepsilon I[u_0]+\frac{2W(2R_0)}{\log 2}\right)^{1-p} \lesssim \varepsilon^{1-p}.
\end{align*} Then, integrating both sides of \eqref{differential inequality for W(R)} and choosing a suitably small $\varepsilon_0>0$, in the subcritical case $p<p_{\Fuj}(Q)$  for any $\varepsilon\in (0,\varepsilon_0]$ we obtain
\begin{align} \label{upper bound proof subcritical}
T^{\left(\frac{1}{p-1}-\frac{Q}{2}\right)(p-1)}\leq C\varepsilon^{-(p-1)} + (2R_0)^{\left(\frac{1}{p-1}-\frac{Q}{2}\right)(p-1)} \leq C\varepsilon^{-(p-1)} ,
\end{align} whereas in the critical case $p=p_{\Fuj}(Q)$ we have
\begin{align} \label{upper bound proof critical}
\log\left(\frac{T}{2R_0}\right)\lesssim (I[u_0])^{1-p} \varepsilon^{-(p-1)} .
\end{align} By \eqref{upper bound proof subcritical} and \eqref{upper bound proof critical} we get the desired estimate in \eqref{upper bound lifespan}. Hence, the statement of the theorem is completely proved.
\end{proof}

\section{Preliminary results}

Goal of this section is to derive some a priori estimates for (local or global in time) solutions to \eqref{semilinear heat Heisenberg}.  
Our approach relies on the following properties of the heat kernel $(t,\eta)\in (0,\infty)\times \mathbf{H}_n \to h_t(\eta)$ on the Heisenberg group (actually, these properties are satisfied in the more general frame of nilpotent Lie group, see \cite{Hor67,Fol75,VSC92}):
\begin{enumerate}
\item the heat kernel is a positive fundamental solution for the heat operator $\partial
_t-\Delta_{\hor}$;
\item the heat kernel is a $\mathcal{C}^\infty((0,\infty)\times\mathbf{H}_n)$ function (this fact follows immediately from the hypoellepticity of $\partial
_t-\Delta_{\hor}$);
\item  the heat kernel satisfies $\|h_t\|_{L^1(\mathbf{H}_n)}=1$  for any $t>0$;
\item the action of the heat semigroup $\{e^{t\Delta_{\hor}}\}_{t>0}$ is given by the convolution
\begin{align*}
e^{t\Delta_{\hor}} v (\eta)\doteq (v\ast h_t) (\eta)=  \int_{\mathbf{H}_n}v(\zeta)h_t(\zeta^{-1}\circ \eta) 	\, d\zeta = \int_{\mathbf{H}_n}v(\eta\circ \zeta^{-1})h_t(\zeta) 	 \, d\zeta,
\end{align*} where $d\zeta$ is the Lebesgue measure on $\mathbb{R}^{2n+1}$ which is also a left and right-invariant Haar measure on the Lie group $\mathbf{H}_n$;
\item there exist two positive constants $c,C$ such that the heat kernel can be estimate as follows:
\begin{align}\label{heat kernel behavior}
c t^{-\frac{Q}{2}}\exp\left( -\frac{C |\eta |_{\mathbf{H}_n}^2}{t}\right)\leq h_t(\eta)\leq C t^{-\frac{Q}{2}}\exp\left( -\frac{c |\eta |_{\mathbf{H}_n}^2}{t}\right)
\end{align} for any $t>0$ and $\eta\in \mathbf{H}_n$.
\end{enumerate}

We underline that several works have been devoted to the study of the heat kernel (fundamental solution of the heat equation) in the Heisenberg group (cf. \cite{Hul76,Gav77,BGG00,GL17} and references therein contained). Our approach will rely basically on the uniform boundedness of the $L^1(\mathbf{H}_n)$-norm of the heat kernel and on the estimate of Gaussian-type \eqref{heat kernel behavior} (for the proof of this result see for example \cite[Theorem 3.12]{KS87}).

\begin{remark} As $\mathbf{H}_n$ is a Carnot group, we may introduce on $\mathbf{H}_n$ the so-called Carnot-Carath\'{e}odory metric $d_{CC}$ as well. Actually, \eqref{heat kernel behavior} is stated in \cite[Theorem VIII.2.9]{VSC92} with $d_{CC}(\eta,0)$ in place of $ |\eta |_{\mathbf{H}_n}$. However, it is well-known that the left-invariant homogeneous norms $d_{CC}(\cdot,0)$ and $|\cdot|_{\mathbf{H}_n}$ are equivalent and, therefore, we may switch them in \eqref{heat kernel behavior} (clearly, up to a modification of the constants $c,C$).
\end{remark}

\begin{remark} From the scaling properties of the heat operator $\partial_t -\Delta_{\hor}$ we can derive a scale-invariance property for the heat kernel (cf. \cite[Theorem 3.1]
{Fol75}). In order to prove this property, let us introduce for any $\lambda>0$ the scaling operators
\begin{align*}
S_\lambda u(t,\eta)& \doteq u(\lambda^{2}t,\delta_{\lambda} (\eta)) ,\\
S_\lambda u_0(\eta)& \doteq u_0(\delta_{\lambda} (\eta)),
\end{align*} where  $\delta_\lambda$ is the anisotropic dilation on $\mathbf{H}_n$. If $u$ solves the homogeneous problem
\begin{align*}
    \begin{cases}
    u_t-\Delta_{\hor}u=0, & \eta\in \mathbf{H}_n, \ t>0, \\
    u(0,\eta) = u_0(\eta), & \eta\in \mathbf{H}_n,
    \end{cases}
\end{align*} then, using the property $$\lambda^2 S_{\lambda^{-1}}(\partial_t-\Delta_{\hor})S_{\lambda}=\partial_t-\Delta_{\hor},$$ we see immediately that $u_{\lambda} \doteq S_{\lambda}u$ solves
\begin{align} \label{hom system u lambda}
    \begin{cases}
    w_t-\Delta_{\hor}w=0, & \eta\in \mathbf{H}_n, \ t>0, \\
    w(0,\eta) = S_{\lambda}u_0(\eta), & \eta\in \mathbf{H}_n.
    \end{cases}
\end{align} Therefore,  we may write $u_\lambda$ in two different ways. On the one hand, we use that $u_\lambda$ solves \eqref{hom system u lambda}
\begin{align*}
    u_{\lambda}(t,\eta) &= \int_{\mathbf{H}_n}h_t(\zeta^{-1}\circ \eta)S_\lambda u_0(\zeta) \, d\zeta =\int_{\mathbf{H}_n}h_t(\zeta^{-1}\circ \eta) \, u_0(\delta_{\lambda} (\zeta)) \, d\zeta.
 \end{align*} On the other hand, we use the fact that $u_\lambda$ is defined through a scaling operator applied to $u$ and, consequently,
 \begin{align*}
    u_{\lambda}(t,\eta) &= \int_{\mathbf{H}_n}h_{\lambda^2 t} (\zeta^{-1}\circ \delta_{\lambda} (\eta)) \, u_0(\zeta) \, d\zeta =  \lambda^{Q} \int_{\mathbf{H}_n}  h_{\lambda^2 t} (\delta_{\lambda}(\zeta)^{-1}\circ \delta_{\lambda} (\eta)) \, u_0(\delta_{\lambda} (\zeta)) \, d\zeta\\
    &=   \int_{\mathbf{H}_n} \lambda^{Q}\,  h_{\lambda^2 t} (\delta_{\lambda}(\zeta^{-1}\circ \eta)) \, u_0(\delta_{\lambda} (\zeta)) \, d\zeta.
\end{align*} As these two expressions coincide for any data $u_0$ and any $\eta\in \mathbf{H}_n$ and $t>0$, then, necessarily we have
\begin{align*} 
h_{ t} (\xi) =  \lambda^{Q} \, h_{\lambda^2 t} (\delta_{\lambda}(\xi))
\end{align*} for any $\lambda>0$ and any $\xi\in \mathbf{H}_n,t>0$. In particular, when $\lambda=t^{-\frac{1}{2}}$ we have
\begin{align} \label{scaling property heat kernel normilized}
h_{ t} (\xi) =  t^{-\frac{Q}{2}} \, h_{1} (\delta_{t^{-1/2}}(\xi))
\end{align} for any $\xi\in \mathbf{H}_n$ and any $t>0$.
\end{remark}

The remaining part of this section is organized as follows: first we prove three preliminary results (cf. Propositions \ref{Prop fund estimate linear term inhom case}, \ref{Prop fund estimate linear term inhom case kappa<Q} and \ref{Prop fund estimate Duhamel term inhom}); hence, we derive some a priori estimates, that will be employed in the proofs of Theorems \ref{Thm GESD} and \ref{Thm LER}. 

\subsection{Estimates for the solution of the homogeneous linear problem}

The next two results will be useful in the treatment of the  solution of the corresponding homogeneous linear problem, when we will apply the contraction principle in order to prove the existence of local (in time) solutions in the subcritical case or the existence of global (in time) small data solutions in the supercritical case, respectively.

\begin{proposition} \label{Prop fund estimate linear term inhom case} Let $\kappa>Q $. Then, for any $t\geq 0$ and $\eta\in \mathbf{H}_n$ the following estimate holds
\begin{align}\label{fundamental estimate linear term inhom case}
e^{t\Delta_{\hor}}\Big(\big(1+|\cdot|_{\mathbf{H}_n}^2\big)^{-\frac{\kappa}{2}}\Big)(\eta)\lesssim \left(1+t+|\eta|_{\mathbf{H}_n}^2\right)^{-\frac{Q}{2}}.
\end{align}
\end{proposition}

\begin{proof}
 In the case $t+|\eta|_{\mathbf{H}_n}^2\leq 1$, it suffices to show that $e^{t\Delta_{\hor}}((1+|\cdot|_{\mathbf{H}_n}^2)^{-\frac{\kappa}{2}})(\eta)$ is bounded. Using the uniform boundedness of the $L^1(\mathbf{H}_n)$ norm of $h_t$, we get immediately
\begin{align*}
e^{t\Delta_{\hor}}\Big(\big(1+|\cdot|_{\mathbf{H}_n}^2\big)^{-\frac{\kappa}{2}}\Big)(\eta)= \int_{\mathbf{H}_n}\left(1+|\eta\circ\zeta^{-1}|^2_{\mathbf{H}_n} \right)^{-\frac{\kappa}{2}}h_t(\zeta)\, d\zeta \lesssim  \int_{\mathbf{H}_n} h_t(\zeta)\, d\zeta \lesssim 1.
\end{align*} Hence, we have to prove \eqref{fundamental estimate linear term inhom case} only when $t+|\eta|_{\mathbf{H}_n}^2\geq 1$. Let us begin by proving it in the case $|\eta|_{\mathbf{H}_n}^2\geq t$. We denote $\eta=(x,y,\tau)$, with $x,y\in \mathbb{R}^n$ and $\tau\in \mathbb{R}$. For that purpose, we shall distinguish among three possible subcases that we label $\tau$-dominant case, $x$-dominant case and $y$-dominant case, respectively. In each case, we fix the greatest number among $8|x|,8|y|$ and $|\tau|^{\frac{1}{2}}$ and we name it correspondingly. The reason for the choice of this nomenclature will be clarified during the proof.

\begin{flushleft}
\emph{$\tau$-dominant case}
\end{flushleft}
  We start in the case in which $\tau $ has a dominant role, namely, when $8|x|\leq|\tau|^{\frac{1}{2}}$ and $8|y|\leq|\tau|^{\frac{1}{2}}$. If these relations are satisfied, then, $|\eta|_{\mathbf{H}_n}^2\simeq |\tau|$. Our goal is to estimate the integral
\begin{align*}
& \int_{\mathbf{H}_n}\left(1+|\eta\circ\zeta^{-1}|^2_{\mathbf{H}_n} \right)^{-\frac{\kappa}{2}}h_t(\zeta)\, d\zeta \\ &  \quad  \simeq \int_{\mathbb{R}^{2n+1}}\left(1+|x-x'|^2+|y-y'|^2+|\tau-\tau'+2(x'\cdot y-x\cdot y')|\right)^{-\frac{\kappa}{2}} t^{-\frac{Q}{2}} e^{-\frac{c}{t}\left(|x'|^2+|y'|^2+|\tau'|\right)}  d(x', y',\tau'). 
\end{align*} Let us consider the following $\tau$-dependent partition of $\mathbb{R}^{2n}$:
\begin{align*}
\mathcal{R}_1(\tau) & \doteq \left\{ (x',y')\in \mathbb{R}^{2n}: |x'|\leq |\tau|^{\frac{1}{2}} \ \mbox{and} \  |y'|\leq |\tau|^{\frac{1}{2}}\right\}, \\
\mathcal{R}_2(\tau) & \doteq \left\{ (x',y')\in \mathbb{R}^{2n}: |x'|\geq |\tau|^{\frac{1}{2}} \ \mbox{and} \  |x'|\geq |y'|\right\},\\
\mathcal{R}_3(\tau) & \doteq \left\{ (x',y')\in \mathbb{R}^{2n}: |y'|\geq |\tau|^{\frac{1}{2}} \ \mbox{and} \  |y'|\geq |x'|\right\}.
\end{align*} Since for $(x',y')\in \mathcal{R}_1(\tau)$ it holds $$|x'\cdot y-x\cdot y'|\leq (|x'|+|y'|)\tfrac{|\tau|^{\frac{1}{2}}}{8}\leq \tfrac{|\tau|}{4},$$ we have that $\tau +2(x'\cdot y-x\cdot y')$ is in the interval $\big[\tau-\frac{|\tau|}{2},\tau+\frac{|\tau|}{2}\big]$. So, for $\tau'\in \big[\tau-\frac{3|\tau|}{4},\tau+\frac{3|\tau|}{4}\big]$ the term $\tau +2(x'\cdot y-x\cdot y')$ belongs to the interval where $\tau'$ runs. Hence, we may not consider a nonnegative lower bound but 0 for $|\tau-\tau' +2(x'\cdot y-x\cdot y')|$. Nonetheless, as $\tau\simeq \tau$ we may estimate $e^{-\frac{c}{2t}|\tau'|}\simeq e^{-\frac{c'}{2t}|\tau|}$ in this region. Combining what we have just remarked, we get
\begin{align*}
&   \int_{\mathcal{R}_{1}(\tau)}\int_{\tau -\frac{3}{4}|\tau|}^{\tau +\frac{3}{4}|\tau|}\left(1+|x-x'|^2+|y-y'|^2+|\tau-\tau'+2(x'\cdot y-x\cdot y')|\right)^{-\frac{\kappa}{2}} t^{-\frac{Q}{2}} e^{-\frac{c}{t}\left(|x'|^2+|y'|^2+|\tau'|\right)}  d\tau' \, d(x', y') \\
&  \quad \lesssim t^{-\frac{Q}{2}} e^{-\frac{c}{4t}|\tau|} \int_{\mathcal{R}_{1}(\tau)} \int_{\tau -\frac{3}{4}|\tau|}^{\tau +\frac{3}{4}|\tau|}\left(1+|x-x'|^2+|y-y'|^2+|\tau-\tau'+2(x'\cdot y-x\cdot y')| \right)^{-\frac{\kappa}{2}} e^{-\frac{c}{t}\left(|x'|^2+|y'|^2\right)} \,  d\tau' \, d(x', y') \\
&  \quad \lesssim t^{-\frac{Q}{2}} e^{-\frac{c}{4t}|\tau|} \int_{\mathbb{R}^{2n+1}}\left(1+|x-x'|^2+|y-y'|^2+|\tau-\tau'+2(x'\cdot y-x\cdot y')| \right)^{-\frac{\kappa}{2}}  \, d(x', y',\tau') \\
&  \quad = t^{-\frac{Q}{2}} e^{-\frac{c}{4t}|\tau|} \int_{\mathbb{R}^{2n+1}}\left(1+|x'|^2+|y'|^2+|\tau'| \right)^{-\frac{\kappa}{2}}  \, d(x', y',\tau') \simeq  t^{-\frac{Q}{2}} e^{-\frac{c}{4t}|\tau|} \int_{\mathbf{H}_{n}}\left(1+|\zeta|^2_{\mathbf{H}_n} \right)^{-\frac{\kappa}{2}}  \, d\zeta \lesssim  t^{-\frac{Q}{2}} e^{-\frac{c}{4t}|\tau|}\\
& \quad \lesssim t^{-\frac{Q}{2}}  \left(1+\tfrac{|\tau|}{t}\right)^{-\frac{Q}{2}} =  (t+|\tau|)^{-\frac{Q}{2}}\simeq \left(1+t+|\eta|_{\mathbf{H}_n}^2\right)^{-\frac{Q}{2}}.
\end{align*} Note that in the second last line of the previous chain of inequalities we employed the condition $\kappa>Q$ to guarantee the boundedness of the integral. More specifically, we applied the analogous version of the integration formula  for radial symmetric functions in the Euclidean space in the case of $|\cdot|_{\mathbf{H}_n}$-symmetric functions on the Heisenberg group (cf. \cite[Proposition 5.4.4]{Lan07}, where this formula is proved in the more general frame of homogeneous Carnot groups). Thus, it results
\begin{align*}
 \int_{\mathbf{H}_{n}}\left(1+|\zeta|^2_{\mathbf{H}_n} \right)^{-\frac{\kappa}{2}}  \, d\zeta \simeq \int_{0}^{\infty}\left(1+\varrho^2\right)^{-\frac{\kappa}{2}}\varrho^{Q-1} \, d\varrho \lesssim   \int_{1}^{\infty}\varrho^{-\kappa+Q-1} \, d\varrho \lesssim  1.
\end{align*} On the other hand, for  $\tau'\not \in \big[\tau-\frac{3|\tau|}{4},\tau+\frac{3|\tau|}{4}\big]$, since  $(x',y')\in \mathcal{R}_1(\tau)$ implies $\tau +2(x'\cdot y-x\cdot y')\in \big[\tau-\frac{|\tau|}{2},\tau+\frac{|\tau|}{2}\big]$ as we have pointed out previously, we may estimate from below $|\tau -\tau'+2(x'\cdot y-x\cdot y')| \geq \frac{|\tau|}{4}$. Therefore,
\begin{align*}
&   \int_{\mathcal{R}_{1}(\tau)} \int_{\tau'\not\in \left[\tau-\frac{3|\tau|}{4},\tau+\frac{3|\tau|}{4}\right]}\left(1+|x-x'|^2+|y-y'|^2+|\tau-\tau'+2(x'\cdot y-x\cdot y')|\right)^{-\frac{\kappa}{2}} t^{-\frac{Q}{2}} e^{-\frac{c}{t}\left(|x'|^2+|y'|^2+|\tau'|\right)}  d\tau' \, d(x', y') \\
& \quad \lesssim \int_{\mathcal{R}_{1}(\tau)} \int_{\tau'\not\in \left[\tau-\frac{3|\tau|}{4},\tau+\frac{3|\tau|}{4}\right]}\left(1+|x-x'|^2+|y-y'|^2+|\tau|\right)^{-\frac{\kappa}{2}} t^{-\frac{Q}{2}} e^{-\frac{c}{t}\left(|x'|^2+|y'|^2+|\tau'|\right)}  d\tau' \, d(x', y') \\
& \quad \lesssim \left(1+|\tau|\right)^{-\frac{\kappa}{2}} \int_{\mathbb{R}^{2n+1}}  t^{-\frac{Q}{2}} e^{-\frac{c}{t}\left(|x'|^2+|y'|^2+|\tau'|\right)}  \, d(x', y',\tau' ) = \left(1+|\tau|\right)^{-\frac{\kappa}{2}} \int_{\mathbb{R}^{2n+1}}  e^{-c\left(|x'|^2+|y'|^2+|\tau'|\right)}  \, d(x', y',\tau' ) \\
& \quad  \lesssim \left(1+|\tau|\right)^{-\frac{\kappa}{2}} \simeq  \left(1+t+|\eta|_{\mathbf{H}_n}^2\right)^{-\frac{\kappa}{2}} \lesssim \left(1+t+|\eta|_{\mathbf{H}_n}^2\right)^{-\frac{Q}{2}}.
\end{align*} Until now, we restrict our considerations to the sub-integral with $(x',y')$  in the region $\mathcal{R}_1(\tau)$. Let us investigate the behavior of the sub-integral with domain $\mathcal{R}_2(\tau)\times \mathbb{R}$ (clearly for the integral over $\mathcal{R}_3(\tau)\times \mathbb{R}$ the situation is completely analogous by switching the role of the variables $x'$ and $y'$). By the definition of $\mathcal{R}_2(\tau)$, we get that $|x'|\geq |\tau|^{\frac{1}{2}}$ for $(x',y')\in \mathcal{R}_2(\tau)$. Hence,
\begin{align*}
&   \int_{\mathcal{R}_{2}(\tau)} \int_{\mathbb{R}}\left(1+|x-x'|^2+|y-y'|^2+|\tau-\tau'+2(x'\cdot y-x\cdot y')|\right)^{-\frac{\kappa}{2}} t^{-\frac{Q}{2}} e^{-\frac{c}{t}\left(|x'|^2+|y'|^2+|\tau'|\right)}  d\tau' \, d(x', y') \\
& \quad \lesssim  t^{-\frac{Q}{2}} e^{-\frac{c}{t}|\tau|} \int_{\mathcal{R}_{2}(\tau)} \int_{\mathbb{R}}\left(1+|x-x'|^2+|y-y'|^2+|\tau-\tau'+2(x'\cdot y-x\cdot y')|\right)^{-\frac{\kappa}{2}}e^{-\frac{c}{t}\left(|y'|^2+|\tau'|\right)}  d\tau' \, d(x', y')  \\
& \quad \lesssim  t^{-\frac{Q}{2}} e^{-\frac{c}{t}|\tau|} \int_{\mathbb{R}^{2n+1}}\left(1+|x-x'|^2+|y-y'|^2+|\tau-\tau'+2(x'\cdot y-x\cdot y')|\right)^{-\frac{\kappa}{2}}  \, d(x', y',\tau')  \\
& \quad \simeq   t^{-\frac{Q}{2}} e^{-\frac{c}{t}|\tau|} \int_{\mathbf{H}_{n}}\left(1+|\zeta|^2_{\mathbf{H}_n}\right)^{-\frac{\kappa}{2}}  \, d\zeta \lesssim  t^{-\frac{Q}{2}} e^{-\frac{c}{t}|\tau|} \lesssim (t+|\tau|)^{-\frac{Q}{2}} \simeq \left(1+t+|\eta|_{\mathbf{H}_n}^2\right)^{-\frac{Q}{2}}.
\end{align*} Summarizing, splitting the integral $(1+|\cdot|^2_{\mathbf{H}_n})^{-\frac{\kappa}{2}}\ast h_t$ on the partition $\{\mathcal{R}_j(\tau)\times \mathbb{R}\}_{1\leq j\leq 3}$ of $\mathbf{H}_n$,  we proved \eqref{fundamental estimate linear term inhom case} in the $\tau$-dominant case.

\begin{flushleft}
\emph{$x$-dominant case}
\end{flushleft}

Let us consider the case $|x|\geq |y|$ and $8|x|\geq |\tau|^{\frac{1}{2}}$.
In this case it suffices to split the integral with respect to $x'$ in three different regions. As for $|x'|\leq \frac{|x|}{2}$ or $|x'|\geq 2|x|$ the estimate $|x-x'|\gtrsim |x|$ holds, we get
\begin{align*}
& \int_{\mathbb{R}^{n+1}}\!\int_{\{2|x'| \leq |x|\}\cup \{  |x'|\geq 2|x| \}}\!\!\left(1+|x-x'|^2+|y-y'|^2+|\tau-\tau'+2(x'\cdot y-x\cdot y')|\right)^{-\frac{\kappa}{2}} t^{-\frac{Q}{2}} e^{-\frac{c}{t}\left(|x'|^2+|y'|^2+|\tau'|\right)}  \, dx' \, d(y',\tau')  \\
& \quad \lesssim  t^{-\frac{Q}{2}} \int_{\mathbb{R}^{n+1}}\left(1+|x|^2+|y-y'|^2\right)^{-\frac{\kappa}{2}}   e^{-\frac{c}{t}\left(|y'|^2+|\tau'|\right)} \, d(y',\tau') \int_{\{2|x'| \leq |x|\}\cup \{  |x'|\geq 2|x| \}} e^{-\frac{c}{t}|x'|^2}  \, dx'   \\
&  \quad \lesssim  t^{-\frac{Q}{2}} \left(1+|x|^2\right)^{-\frac{\kappa}{2}}   \int_{\mathbb{R}^{n+1}} e^{-\frac{c}{t}\left(|y'|^2+|\tau'|\right)} \, d(y',\tau')  \int_{\mathbb{R}^n} e^{-\frac{c}{t}|x'|^2}  \, dx'  \\
& \quad \lesssim \left(1+|x|^2\right)^{-\frac{\kappa}{2}}   \int_{\mathbb{R}^{n+1}} e^{-c\left(|y'|^2+|\tau'|\right)} \, d(y',\tau') \int_{\mathbb{R}^n} e^{-c|x'|^2}  \, dx'  \lesssim \left(1+|x|^2\right)^{-\frac{\kappa}{2}} \lesssim  \left(1+t+|\eta|_{\mathbf{H}_n}^2\right)^{-\frac{Q}{2}}.
\end{align*} Otherwise, for $\frac{|x|}{2}\leq |x'| \leq 2|x|$ we use the exponential decay as follows
\begin{align*}
& \int_{\mathbb{R}^{n+1}}\int_{\big\{\frac{|x|}{2} \leq |x'|\leq 2|x| \big\}}\left(1+|x-x'|^2+|y-y'|^2+|\tau-\tau'+2(x'\cdot y-x\cdot y')|\right)^{-\frac{\kappa}{2}} t^{-\frac{Q}{2}} e^{-\frac{c}{t}\left(|x'|^2+|y'|^2+|\tau'|\right)}  \, dx' \, d(y',\tau')  \\
& \quad \lesssim t^{-\frac{Q}{2}} e^{-\frac{c}{4t}|x|^2} \! \!\int_{\mathbb{R}^{n+1}} \int_{\big\{\frac{|x|}{2} \leq |x'|\leq 2|x| \big\}}\!\!\left(1+|x-x'|^2+|y-y'|^2+|\tau-\tau'+2(x'\cdot y-x\cdot y')|\right)^{-\frac{\kappa}{2}}  \! e^{-\frac{c}{t}\left(|y'|^2+|\tau'|\right)}  \, dx' \, d(y',\tau')  \\
& \quad \lesssim t^{-\frac{Q}{2}} e^{-\frac{c}{4t}|x|^2} \int_{\mathbb{R}^{2n+1}} \left(1+|x-x'|^2+|y-y'|^2+|\tau-\tau'+2(x'\cdot y-x\cdot y')|\right)^{-\frac{\kappa}{2}}   \, d(x',y',\tau') \\
& \quad = t^{-\frac{Q}{2}} e^{-\frac{c}{4t}|x|^2} \int_{\mathbb{R}^{2n+1}} \left(1+|x'|^2+|y'|^2+|\tau'|\right)^{-\frac{\kappa}{2}}   \, d(x',y',\tau') \simeq  t^{-\frac{Q}{2}} e^{-\frac{c}{4t}|x|^2} \int_{\mathbf{H}_{n}} \left(1+|\zeta|^2_{\mathbf{H}_n}\right)^{-\frac{\kappa}{2}}   \, d\zeta\\
& \quad  \lesssim t^{-\frac{Q}{2}} e^{-\frac{c}{4t}|x|^2}  \lesssim \left(t+|x|^2\right)^{-\frac{Q}{2}} \simeq  \left(1+t+|\eta|_{\mathbf{H}_n}^2\right)^{-\frac{Q}{2}}.
\end{align*} Also, we proved \eqref{fundamental estimate linear term inhom case} in the $x$-dominant case.

\begin{flushleft}
\emph{$y$-dominant case}
\end{flushleft}

In this case $|y|\geq |x|$ and $8|y|\geq |\tau|^{\frac{1}{2}}$. We can proceed analogously as in the previous case by splitting the domain of integration for $y'$ into $\{|y'|\leq \frac{|y|}{2}\}$, $\{\frac{|y|}{2}\leq |y'|\leq 2|y|\}$ and $\{ |y'|\geq 2|y|\}$ and by swapping the role of $(x,x')$ and $(y,y')$.


\vspace{0.3cm}

So far, we dealt with the case in which we have the inequality $|\eta |_{\mathbf{H}_n}^2\geq t$. When $t$ is dominant, that is, the reverse inequality $|\eta |_{\mathbf{H}_n}^2\leq t$ holds, then \eqref{fundamental estimate linear term inhom case} follows by the estimate $\|h_t\|_{L^\infty(\mathbf{H}_n)}\lesssim t^{-\frac{Q}{2}}$. More precisely,
\begin{align*}
e^{t\Delta_{\hor}}\Big(\big(1+|\cdot|_{\mathbf{H}_n}^2\big)^{-\frac{\kappa}{2}}\Big)(\eta) &= \int_{\mathbf{H}_n}\left(1+|\zeta|^2_{\mathbf{H}_n}\right)^{-\frac{\kappa}{2}} h_t(\zeta^{-1}\circ \eta) \, d\zeta  \lesssim t^{-\frac{Q}{2}} \int_{\mathbf{H}_n}\left(1+|\zeta|^2_{\mathbf{H}_n}\right)^{-\frac{\kappa}{2}} \, d\zeta \\ &  \lesssim  t^{-\frac{Q}{2}} \simeq \left(1+t+|\eta|_{\mathbf{H}_n}^2\right)^{-\frac{Q}{2}}.
\end{align*}  Hence, we completed the proof in all possible subcases.
\end{proof}

\begin{proposition} \label{Prop fund estimate linear term inhom case kappa<Q} Let $\kappa\in (0,Q )$. Then, for any $t\geq 0$ and $\eta\in \mathbf{H}_n$ the following estimate holds
\begin{align}\label{fundamental estimate linear term inhom case kappa<Q}
e^{t\Delta_{\hor}}\Big(\big(1+|\cdot|_{\mathbf{H}_n}^2\big)^{-\frac{\kappa}{2}}\Big)(\eta)\lesssim \left(1+t+|\eta|_{\mathbf{H}_n}^2\right)^{-\frac{\kappa}{2}}.
\end{align}
\end{proposition}

\begin{proof}
As in the proof of Proposition \ref{Prop fund estimate linear term inhom case}, we may restrict ourselves to consider the case  $t+|\eta|_{\mathbf{H}_n}^2\geq 1$ (otherwise, we employ again the uniform $L^1(\mathbf{H}_n)$ boundedness of the heat kernel). Actually, in this case it is possible to show the validity of a stronger estimate, namely,
\begin{align}\label{fundamental estimate linear term}
e^{t\Delta_{\hor}}\left(|\cdot|_{\mathbf{H}_n}^{-\kappa}\right)(\eta)\lesssim \left(t+|\eta|_{\mathbf{H}_n}^2\right)^{-\frac{\kappa}{2}} \simeq \left(1+t+|\eta|_{\mathbf{H}_n}^2\right)^{-\frac{\kappa}{2}}.
\end{align} Indeed, if \eqref{fundamental estimate linear term} holds, then, using the positivity of the heat kernel, by the monotonicity of $e^{t\Delta_{\hor}}$ we get immediately \eqref{fundamental estimate linear term inhom case kappa<Q}. The advantage in considering this \emph{homogeneous} inequality rather than \eqref{fundamental estimate linear term inhom case kappa<Q} is that it suffices to show \eqref{fundamental estimate linear term} for $t=1$, namely,
\begin{align} \label{estimate linear term normalized}
e^{\Delta_{\hor}}\left(|\cdot|_{\mathbf{H}_n}^{-\kappa}\right)(\eta)
\lesssim \left(1+|\eta|_{\mathbf{H}_n}^2\right)^{-\frac{\kappa}{2}} \quad \mbox{for any} \ \ \eta\in \mathbf{H}_n.
\end{align} Indeed, by \eqref{scaling property heat kernel normilized} it follows
\begin{align*}
\big(|\cdot|_{\mathbf{H}_n}^{-\kappa}\ast h_t \big)(\eta) &=  \int_{\mathbf{H}_n} h_t(\zeta^{-1} \circ \eta) |\zeta|_{\mathbf{H}_n}^{-\kappa} \, d\zeta  = t^{-\frac{Q}{2}}  \int_{\mathbf{H}_n} h_1\left(\delta_{t^{-1/2}}(\zeta^{-1} \circ \eta)\right) |\zeta|_{\mathbf{H}_n}^{-\kappa} \, d\zeta   \\
& = t^{-\frac{Q}{2}}  \int_{\mathbf{H}_n} h_1\left(\delta_{t^{-1/2}}(\zeta)^{-1} \circ \delta_{t^{-1/2}}(\eta)\right) |\zeta|_{\mathbf{H}_n}^{-\kappa} \, d\zeta    =  \int_{\mathbf{H}_n} h_1\left(\xi^{-1} \circ \delta_{t^{-1/2}}(\eta)\right) |\delta_{t^{1/2}}(\xi)|_{\mathbf{H}_n}^{-\kappa} \, d\xi   \\
& =  t^{-\frac{\kappa}{2}}\int_{\mathbf{H}_n} h_1\left(\xi^{-1} \circ \delta_{t^{-1/2}}(\eta)\right) |\xi|_{\mathbf{H}_n}^{-\kappa} \, d\xi  = t^{-\frac{\kappa}{2}} \big(|\cdot|_{\mathbf{H}_n}^{-\kappa}\ast h_1 \big)(\delta_{t^{-1/2}}(\eta)),
\end{align*}
 where we preformed the change of variables $\xi= \delta_{t^{-1/2}}(\zeta)$ and we used the fact that the anisotropic dilation $\delta_r$ is an isomorphism of Lie group on $\mathbf{H}_n$ with $(\delta_r)^{-1}=\delta_{r^{-1}}$ and the homogeneity of degree 1 for $|\cdot|_{\mathbf{H}_n}$ with respect to anisotropic dilations. Therefore, if we prove \eqref{estimate linear term normalized}, then, it follows that
\begin{align*}
e^{t\Delta_{\hor}}\left(|\cdot|_{\mathbf{H}_n}^{-\kappa}\right)(\eta)& = \big(|\cdot|_{\mathbf{H}_n}^{-\kappa}\ast h_t \big)(\eta) = t^{-\frac{\kappa}{2}} \big(|\cdot|_{\mathbf{H}_n}^{-\kappa}\ast h_1 \big)(\delta_{t^{-1/2}}(\eta))= t^{-\frac{\kappa}{2}} e^{\Delta_{\hor}}\left(|\cdot|_{\mathbf{H}_n}^{-\kappa}\right)(\delta_{t^{-1/2}}(\eta)) \\
& \lesssim t^{-\frac{\kappa}{2}}  \left(1+|\delta_{t^{-1/2}}(\eta)|_{\mathbf{H}_n}^2\right)^{-\frac{\kappa}{2}} =t^{-\frac{\kappa}{2}}  \left(1+t^{-1}|\eta|_{\mathbf{H}_n}^2\right)^{-\frac{\kappa}{2}} =\left(t+|\eta|_{\mathbf{H}_n}^2\right)^{-\frac{\kappa}{2}},
\end{align*} which is exactly \eqref{fundamental estimate linear term}.

So, we prove now \eqref{estimate linear term normalized}. Note that \eqref{triangular ineq gauge} implies the validity of the reverse triangular inequality $$\big||\zeta|_{\mathbf{H}_n}-|\eta|_{\mathbf{H}_n}\big|\leq |\zeta^{-1}\circ \eta|_{\mathbf{H}_n}.$$ We shall employ this fact to split the domain of the integral in the left-hand side of \eqref{estimate linear term normalized} in different zones. Let us begin with the case $|\eta|_{\mathbf{H}_n}\leq 1$. Clearly, in this case it sufficient to show that $e^{\Delta_{\hor}}\left(|\cdot|_{\mathbf{H}_n}^{-\kappa}\right)(\eta)$ is bounded. We split the estimate as follows:
\begin{align*}
e^{\Delta_{\hor}}\left(|\cdot|_{\mathbf{H}_n}^{-\kappa}\right)(\eta)& \lesssim \int_{\mathbf{H}_n} e^{-c|\zeta^{-1} \circ \eta|_{\mathbf{H}_n}^2} |\zeta|_{\mathbf{H}_n}^{-\kappa} \, d\zeta \\ & \lesssim \int_{|\zeta|_{\mathbf{H}_n}\geq 2} e^{-c|\zeta^{-1} \circ \eta|_{\mathbf{H}_n}^2} |\zeta|_{\mathbf{H}_n}^{-\kappa} \, d\zeta +  \int_{|\zeta|_{\mathbf{H}_n}\leq 2} e^{-c|\zeta^{-1} \circ \eta|_{\mathbf{H}_n}^2} |\zeta|_{\mathbf{H}_n}^{-\kappa} \, d\zeta.
\end{align*}
Let us begin with the estimate for the integral away from the origin. Since $|\zeta|_{\mathbf{H}_n}\geq 2 \geq 2|\eta|_{\mathbf{H}_n}$, then, $|\zeta^{-1} \circ \eta|_{\mathbf{H}_n}\geq \frac{1}{2}|\zeta |_{\mathbf{H}_n}$. Therefore,
\begin{align*}
 \int_{|\zeta|_{\mathbf{H}_n}\geq 2} e^{-c|\zeta^{-1} \circ \eta|_{\mathbf{H}_n}^2} |\zeta|_{\mathbf{H}_n}^{-\kappa} \, d\zeta  & \leq  \int_{|\zeta|_{\mathbf{H}_n}\geq 2} e^{-\frac{c}{4}|\zeta|_{\mathbf{H}_n}^2} |\zeta|_{\mathbf{H}_n}^{-\kappa} \, d\zeta \\
 & \leq  2^{-\kappa} \int_{|\zeta|_{\mathbf{H}_n}\geq 2} e^{-\frac{c}{4}|\zeta|_{\mathbf{H}_n}^2} \, d\zeta  \lesssim 1,
\end{align*} where in the last step we used the fact that
\begin{align*}
 \int_{\mathbf{H}_n} e^{-\frac{c}{4}|\zeta|_{\mathbf{H}_n}^2} \, d\zeta  \leq \int_{\mathbb{R}^n}e^{-\frac{c}{8}|x'|^2} \, dx' \int_{\mathbb{R}^n}e^{-\frac{c}{8}|y'|^2} \, dy' \int_{\mathbb{R}}e^{-\frac{c}{8}|\tau'|} \, d\tau'<\infty.
\end{align*} We consider now the integral close to the origin, where the integrand is singular. We have
\begin{align*}
\int_{|\zeta|_{\mathbf{H}_n}\leq 2} e^{-c|\zeta^{-1} \circ \eta|_{\mathbf{H}_n}^2} |\zeta|_{\mathbf{H}_n}^{-\kappa} \, d\zeta & \leq \int_{|\zeta|_{\mathbf{H}_n}\leq 2}  |\zeta|_{\mathbf{H}_n}^{-\kappa} \, d\zeta.
\end{align*}
Using Young's inequality
\begin{align*}
\prod_{j=1}^k a_j\leq \sum_{j=1}^k \frac{a_j^{p_j}}{p_j}
\end{align*} for any positive $a_1,\cdots ,a_k$ under the constrain $\sum_{j=1}^k \frac{1}{p_j}=1$ (this inequality is a straightforward consequence of the concavity of the logarithmic function), we may estimate
\begin{align*}
|\zeta |_{\mathbf{H}_n} \simeq |x'|+|y'|+|\tau'|^{\frac{1}{2}}\gtrsim |x'|^{\frac{n}{Q}}  |y'|^{\frac{n}{Q}}  |\tau'|^{\frac{1}{Q}} \qquad \mbox{for any} \ \ \zeta=(x',y',\tau')\in \mathbf{H}_n.
\end{align*}  Therefore, as $|\zeta|_{\mathbf{H}_n}\leq 2$ implies $|x'|\leq 2, |y'|\leq 2$ and $ |\tau'|\leq 4$, we get
\begin{align*}
\int_{|\zeta|_{\mathbf{H}_n}\leq 2} e^{-c|\zeta^{-1} \circ \eta|_{\mathbf{H}_n}^2} |\zeta|_{\mathbf{H}_n}^{-\kappa} \, d\zeta & \lesssim  \int_{|x'|\leq 2} |x'|^{-\frac{\kappa n}{Q}}\, dx'  \int_{|y'|\leq 2} |y'|^{-\frac{\kappa n}{Q}} \, dy'  \int_{|\tau'|\leq 4} |\tau'|^{-\frac{\kappa }{Q}} \, d\tau' \lesssim 1,
\end{align*} where we used the condition $\kappa <Q$ in order to guarantee the integrability of the singularities in each integral with respect to $x',y'$ and $\tau'$, respectively. So, we proved \eqref{estimate linear term normalized} in the case $|\eta|_{\mathbf{H}_n}\leq 2$. We consider now the case $|\eta|_{\mathbf{H}_n}\geq 2$. In this case, we split the domain of integration in three zones, namely,
\begin{align*}
e^{\Delta_{\hor}}\left(|\cdot|_{\mathbf{H}_n}^{-\kappa}\right)(\eta)& \lesssim \int_{\mathbf{H}_n} e^{-c|\zeta|_{\mathbf{H}_n}^2} | \eta \circ \zeta^{-1} |_{\mathbf{H}_n}^{-\kappa} \, d\zeta \\ & \lesssim \left(  \int_{|\zeta|_{\mathbf{H}_n}\leq \frac{1}{2}|\eta|_{\mathbf{H}_n}} + \int_{|\zeta|_{\mathbf{H}_n}\geq 2|\eta|_{\mathbf{H}_n}}  +  \int_{\frac{1}{2}|\eta|_{\mathbf{H}_n}\leq |\zeta|_{\mathbf{H}_n}\leq 2|\eta|_{\mathbf{H}_n}}\right) e^{-c|\zeta|_{\mathbf{H}_n}^2} | \eta \circ \zeta^{-1} |_{\mathbf{H}_n}^{-\kappa} \, d\zeta.
\end{align*} Using the reverse triangular inequality, we find $| \eta \circ \zeta^{-1} |_{\mathbf{H}_n}\geq \frac{1}{2} | \eta |_{\mathbf{H}_n}$ in the region $|\zeta|_{\mathbf{H}_n}\leq \frac{1}{2}|\eta|_{\mathbf{H}_n}$. Then,
\begin{align*}
\int_{|\zeta|_{\mathbf{H}_n}\leq \frac{1}{2}|\eta|_{\mathbf{H}_n}} e^{-c|\zeta|_{\mathbf{H}_n}^2} | \eta \circ \zeta^{-1} |_{\mathbf{H}_n}^{-\kappa} \, d\zeta & \leq 2^\kappa | \eta |_{\mathbf{H}_n}^{-\kappa}  \int_{|\zeta|_{\mathbf{H}_n}\leq \frac{1}{2}|\eta|_{\mathbf{H}_n}} e^{-c|\zeta|_{\mathbf{H}_n}^2} \, d\zeta \\
& \lesssim  | \eta |_{\mathbf{H}_n}^{-\kappa}  \int_{\mathbf{H}_n} e^{-c|\zeta|_{\mathbf{H}_n}^2} \, d\zeta \lesssim  | \eta |_{\mathbf{H}_n}^{-\kappa}  \simeq \left(1+ | \eta |_{\mathbf{H}_n}^2\right)^{-\frac{\kappa}{2}}.
\end{align*} For $|\zeta|_{\mathbf{H}_n}\geq 2|\eta|_{\mathbf{H}_n}$, we have $ | \eta \circ \zeta^{-1} |_{\mathbf{H}_n}\geq 
|\eta|_{\mathbf{H}_n}$. Thus, proceeding analogously as in the estimate of the previous integral, we obtain
\begin{align*}
\int_{|\zeta|_{\mathbf{H}_n}\geq 2|\eta|_{\mathbf{H}_n}} e^{-c|\zeta|_{\mathbf{H}_n}^2} | \eta \circ \zeta^{-1} |_{\mathbf{H}_n}^{-\kappa} \, d\zeta & \leq  | \eta  |_{\mathbf{H}_n}^{-\kappa}  \int_{|\zeta|_{\mathbf{H}_n}\geq 2|\eta|_{\mathbf{H}_n}} e^{-c|\zeta|_{\mathbf{H}_n}^2}\, d\zeta \lesssim   \left(1+ | \eta |_{\mathbf{H}_n}^2\right)^{-\frac{\kappa}{2}}.
\end{align*} Finally,
\begin{align*}
\int_{\frac{1}{2}|\eta|_{\mathbf{H}_n}\leq |\zeta|_{\mathbf{H}_n}\leq 2|\eta|_{\mathbf{H}_n}}  e^{-c|\zeta|_{\mathbf{H}_n}^2} | \eta \circ \zeta^{-1} |_{\mathbf{H}_n}^{-\kappa} \, d\zeta & \leq   e^{-\frac{c}{4}|\eta|_{\mathbf{H}_n}^2} \int_{\frac{1}{2}|\eta|_{\mathbf{H}_n}\leq |\zeta|_{\mathbf{H}_n}\leq 2|\eta|_{\mathbf{H}_n}}  | \eta \circ \zeta^{-1} |_{\mathbf{H}_n}^{-\kappa} \, d\zeta \\
 & \leq   e^{-\frac{c}{4}|\eta|_{\mathbf{H}_n}^2} \int_{|\xi|_{\mathbf{H}_n}\leq 3 |\eta|_{\mathbf{H}_n}}  | \xi |_{\mathbf{H}_n}^{-\kappa} \, d\zeta,
\end{align*} where we carried out the change of variables $\xi = \zeta \circ \eta^{-1}$ in the last inequality.
 If we denote $\xi=(x',y',\tau')$, then, $$|\xi|_{\mathbf{H}_n}^{-\kappa}\lesssim |x'|^{-\frac{\kappa n}{Q}} |y'|^{-\frac{\kappa n}{Q}} |\tau'|^{-\frac{\kappa }{Q}}.$$ Moreover, $|\xi|_{\mathbf{H}_n}\leq 3 |\eta|_{\mathbf{H}_n}$ implies $|x'|\leq 3|\eta|_{\mathbf{H}_n},  |y'|\leq 3|\eta|_{\mathbf{H}_n} $ and $|\tau'|\leq 9|\eta|_{\mathbf{H}_n}^2$. Also,
\begin{align*}
\int_{|\xi|_{\mathbf{H}_n}\leq 3 |\eta|_{\mathbf{H}_n}}  | \xi |_{\mathbf{H}_n}^{-\kappa} \, d\zeta & \lesssim \int_{|x'|\leq 3 |\eta|_{\mathbf{H}_n}} |x'|^{-\frac{\kappa n}{Q}} \, dx' \int_{|y'|\leq 3 |\eta|_{\mathbf{H}_n}} |y'|^{-\frac{\kappa n}{Q}} \, dy' \int_{|\tau'|\leq 9 |\eta|_{\mathbf{H}_n}^2} |\tau'|^{-\frac{\kappa }{Q}} \, d\tau'   \\
& \lesssim \left( \int_0^{ 3 |\eta|_{\mathbf{H}_n}} r^{-\frac{\kappa n}{Q}+n-1} \, dr \right)^2 \int_0^{9 |\eta|_{\mathbf{H}_n}^2}r^{-\frac{\kappa }{Q}} \, dr  \lesssim | \eta|_{\mathbf{H}_n}^{2n\left(1-\frac{\kappa }{Q}\right)}  | \eta|_{\mathbf{H}_n}^{2\left(1-\frac{\kappa }{Q}\right)} = | \eta|_{\mathbf{H}_n}^{Q-\kappa } ,
\end{align*} where we employed again the condition $\kappa<Q$ in order to estimate the singular integrals in the last inequality. Consequently, we end up with the estimate
\begin{align*}
\int_{\frac{1}{2}|\eta|_{\mathbf{H}_n}\leq |\zeta|_{\mathbf{H}_n}\leq 2|\eta|_{\mathbf{H}_n}}  e^{-c|\zeta|_{\mathbf{H}_n}^2} | \eta \circ \zeta^{-1} |_{\mathbf{H}_n}^{-\kappa} \, d\zeta & \leq e^{-\frac{c}{4}|\eta|_{\mathbf{H}_n}^2} | \eta|_{\mathbf{H}_n}^{Q } | \eta|_{\mathbf{H}_n}^{-\kappa } \lesssim   | \eta|_{\mathbf{H}_n}^{-\kappa } \simeq   \left(1+ | \eta |_{\mathbf{H}_n}^2\right)^{-\frac{\kappa}{2}}.
\end{align*} Summarizing, if we combine the estimates for three subintegrals we find \eqref{estimate linear term normalized} in the case $|\eta|_{\mathbf{H}_n}\geq 2$ as well. This completes the proof.
\end{proof}

\begin{remark} In the statement of Propositions \ref{Prop fund estimate linear term inhom case} and \ref{Prop fund estimate linear term inhom case kappa<Q}  we considered the case $0<\kappa\neq Q$. It is possible to consider the case $\kappa=Q$ as well, provided that a further factor of logarithmic type is included in \eqref{fundamental estimate linear term inhom case}.  However, since in the treatment of the semilinear Cauchy problem we will apply these estimates just in the case $\kappa\neq Q $, we skip further details.
\end{remark}

\subsection{Estimates for Duhamel's integral term}

The next result will be employed in order to deal with Duhamel's integral term in the integral formulation of the Cauchy problem \eqref{semilinear heat Heisenberg} (cf. Definition \ref{Def mild sol} in Section \ref{Section main results}).

\begin{proposition} \label{Prop fund estimate Duhamel term inhom} Let $0<\alpha\leq 1+\frac{Q}{2}$. Then, for any $t\geq 0$ and $\eta\in \mathbf{H}_n$ the following estimate holds
\begin{align} \label{fundamental estimate Duhamel term inhom}
\int_0^t e^{(t-s)\Delta_{\hor}}\left(1+s+|\cdot|_{\mathbf{H}_n}^2\right)^{-\alpha}\, ds \, (\eta) \lesssim \begin{cases} t\left(1+t+|\eta|_{\mathbf{H}_n}^2\right)^{-\alpha} & \mbox{if} \ \ \alpha\in \big(0,1+\tfrac{Q}{2}\big), \\  t\left(1+t+|\eta|_{\mathbf{H}_n}^2\right)^{-\alpha} \log(e+t) & \mbox{if} \ \ \alpha=1+\tfrac{Q}{2}.\end{cases}
\end{align}
\end{proposition}


\begin{proof} Let us denote
\begin{align*}
I & \doteq \int_0^t e^{(t-s)\Delta_{\hor}}\left(1+s+|\cdot|_{\mathbf{H}_n}^2\right)^{-\alpha}\, ds \, (\eta) =  \int_0^t \int_{\mathbf{H}_n} h_{t-s}(\zeta^{-1}\circ \eta)\left(1+s+|\zeta|_{\mathbf{H}_n}^2\right)^{-\alpha} \, d\zeta \, ds.
\end{align*}
Since the weight function $\left(1+t+|\eta|^2_{\mathbf{H}_n}\right)^{\alpha}$ behaves as a constant for small values of $t+|\eta|^2_{\mathbf{H}_n}$, 
in order to prove \eqref{fundamental estimate Duhamel term inhom}, we will consider separately the case $t+|\eta|^2_{\mathbf{H}_n}\lesssim 1$ and the case $t+|\eta|^2_{\mathbf{H}_n}\gtrsim 1$.

\begin{flushleft}
\emph{Case $t+|\eta|^2_{\mathbf{H}_n}\lesssim 1$}
\end{flushleft}
In this case, it suffices to show that $I\lesssim t$, since $1+t+|\eta|^2_{\mathbf{H}_n}\simeq 1$. By using $\|h_1\|_{L^1(\mathbf{H}_n)}=1$ and $\alpha>0$, we find immediately
\begin{align*}
I &  =  \int_0^t \int_{\mathbf{H}_n} h_{t-s}(\zeta)\left(1+s+|\eta\circ  \zeta^{-1}|_{\mathbf{H}_n}^2\right)^{-\alpha} \, d\zeta \, ds \leq \int_0^t \int_{\mathbf{H}_n} h_{t-s}(\zeta) \, d\zeta \, ds  \leq t.
\end{align*}

\begin{flushleft}
\emph{Case $t+|\eta|^2_{\mathbf{H}_n}\gtrsim 1$}
\end{flushleft}
In this case, it is useful to split the integral $I$ in two parts, namely,
\begin{align*}
I_1 & \doteq  \int_{\frac{t}{2}}^t \int_{\mathbf{H}_n} h_{t-s}(\zeta)\left(1+s+|\eta\circ\zeta^{-1}|^2_{\mathbf{H}_n}\right)^{-\alpha} \, d\zeta \, ds, \\
I_2 & \doteq  \int_0^{\frac{t}{2}} \int_{\mathbf{H}_n}h_{t-s}(\zeta)\left(1+s+|\eta\circ\zeta^{-1}|^2_{\mathbf{H}_n}\right)^{-\alpha}  \, d\zeta \, ds.
\end{align*}
We begin by estimating $I_1$. Let us remark that for $\zeta\in \mathbf{H}_n$ such that $|\zeta|_{\mathbf{H}_n}\leq \frac{|\eta|_{\mathbf{H}_n}}{2}$ or $|\zeta|_{\mathbf{H}_n}\geq 2|\eta|_{\mathbf{H}_n}$ the inequality $|\eta\circ \zeta^{-1}|_{\mathbf{H}_n}\gtrsim |\eta|_{\mathbf{H}_n}$ holds. Also,
\begin{align*}
 \int_{\frac{t}{2}}^t & \int_{\{2|\zeta|_{\mathbf{H}_n}\leq |\eta|_{\mathbf{H}_n}\} \cup \{|\zeta|_{\mathbf{H}_n}\geq 2|\eta|_{\mathbf{H}_n}\}} h_{t-s}(\zeta)\left(1+s+|\eta\circ\zeta^{-1}|^2_{\mathbf{H}_n}\right)^{-\alpha} \, d\zeta \, ds \\
& \quad \qquad \lesssim  \int_{\frac{t}{2}}^t \left(1+s+|\eta|^2_{\mathbf{H}_n}\right)^{-\alpha}  \int_{\{2|\zeta|_{\mathbf{H}_n}\leq |\eta|_{\mathbf{H}_n}\} \cup \{|\zeta|_{\mathbf{H}_n}\geq 2|\eta|_{\mathbf{H}_n}\}} h_{t-s}(\zeta)\, d\zeta \, ds \\
& \quad \qquad \lesssim \left(1+t+|\eta|^2_{\mathbf{H}_n}\right)^{-\alpha} \int_{\frac{t}{2}}^t   \int_{\mathbf{H}_n} h_{t-s}(\zeta)\, d\zeta \, ds \lesssim t \left(1+t+|\eta|^2_{\mathbf{H}_n}\right)^{-\alpha},
\end{align*} where in the last step we used the uniform boundedness of the $L^1(\mathbf{H}_n)$-norm of the heat kernel. On the other hand, we have
\begin{align*}
\int_{\frac{t}{2}}^t & \int_{\{2^{-1} |\eta|_{\mathbf{H}_n} \,\leq |\zeta|_{\mathbf{H}_n}\leq 2 |\eta|_{\mathbf{H}_n}\} } h_{t-s}(\zeta)\left(1+s+|\eta\circ\zeta^{-1}|^2_{\mathbf{H}_n}\right)^{-\alpha} \, d\zeta \, ds \\
& \quad \lesssim \int_{\frac{t}{2}}^t  \int_{\{2^{-1} |\eta|_{\mathbf{H}_n} \,\leq |\zeta|_{\mathbf{H}_n}\leq 2 |\eta|_{\mathbf{H}_n}\} } (t-s)^{-\frac{Q}{2}}e^{-\frac{c}{t-s}|\zeta|^2_{\mathbf{H}_n}}\left(1+s+|\eta\circ\zeta^{-1}|^2_{\mathbf{H}_n}\right)^{-\alpha} \, d\zeta \, ds \\
& \quad \lesssim \left(1+t\right)^{-\alpha} \int_{\frac{t}{2}}^t  \int_{\{2^{-1} |\eta|_{\mathbf{H}_n} \,\leq |\zeta|_{\mathbf{H}_n}\leq 2 |\eta|_{\mathbf{H}_n}\} } (t-s)^{-\frac{Q}{2}}e^{-\frac{c}{t-s}|\zeta|^2_{\mathbf{H}_n}} \, d\zeta \, ds \\
& \quad \lesssim \left(1+t\right)^{-\alpha} \int_{\frac{t}{2}}^t e^{-\frac{c}{8(t-s)}|\eta|^2_{\mathbf{H}_n}}   \int_{\{ 2^{-1} |\eta|_{\mathbf{H}_n} \,\leq |\zeta|_{\mathbf{H}_n}\leq 2 |\eta|_{\mathbf{H}_n}\} } (t-s)^{-\frac{Q}{2}}e^{-\frac{c}{2(t-s)}|\zeta|^2_{\mathbf{H}_n}} \, d\zeta \, ds \\
& \quad \lesssim \left(1+t\right)^{-\alpha} e^{-\frac{c}{4t}|\eta|^2_{\mathbf{H}_n}}    \int_{\frac{t}{2}}^t \int_{\mathbf{H}_n} (t-s)^{-\frac{Q}{2}}e^{-\frac{c}{2(t-s)}|\zeta|^2_{\mathbf{H}_n}} \, d\zeta \, ds .
\end{align*}
Performing the change of variables $\xi=\delta_{(t-s)^{-1/2}}(\zeta)$, it results
\begin{align*}
\int_{\frac{t}{2}}^t & \int_{\{ 2^{-1}|\eta|_{\mathbf{H}_n} \,\leq |\zeta|_{\mathbf{H}_n}\leq 2 |\eta|_{\mathbf{H}_n}\} } h_{t-s}(\zeta)\left(1+s+|\eta\circ\zeta^{-1}|^2_{\mathbf{H}_n}\right)^{-\alpha} \, d\zeta \, ds \\
& \quad \lesssim  \left(1+t\right)^{-\alpha} e^{-\frac{c}{4t}|\eta|^2_{\mathbf{H}_n}}    \int_{\frac{t}{2}}^t \int_{\mathbf{H}_n}e^{-\frac{c}{2}|\zeta|^2_{\mathbf{H}_n}} \, d\zeta \, ds  \lesssim t \left(1+t\right)^{-\alpha} e^{-\frac{c}{4t}|\eta|^2_{\mathbf{H}_n}} \\
& \quad = t \left(1+t+|\eta|^2_{\mathbf{H}_n}\right)^{-\alpha} \left(1+\tfrac{|\eta|^2_{\mathbf{H}_n}}{1+t}\right)^{\alpha} e^{-\frac{c}{4t}|\eta|^2_{\mathbf{H}_n}} \lesssim t \left(1+t+|\eta|^2_{\mathbf{H}_n}\right)^{-\alpha} \left(1+\tfrac{|\eta|^2_{\mathbf{H}_n}}{t}\right)^{\alpha} e^{-\frac{c}{4t}|\eta|^2_{\mathbf{H}_n}}  \\
& \quad \lesssim t \left(1+t+|\eta|^2_{\mathbf{H}_n}\right)^{-\alpha}.
\end{align*} Summarizing, we  proved that $I_1\lesssim t \left(1+t+|\eta|^2_{\mathbf{H}_n}\right)^{-\alpha} $. Note that in order to derive this estimate for $I_1$ we did not consider separately the case $\alpha<1+Q/2$ from the limit case $\alpha=1+Q/2$.  The next step is to prove the validity of the same type of estimate but now for the integral $I_2$. As in the previous case, we need to divide $\mathbf{H}_n$ in different regions. Let us begin with the estimate of integral on the sub-region $\{\zeta\in \mathbf{H}_n: |\zeta|_{\mathbf{H}_n}\leq 2^{-1}  |\eta|_{\mathbf{H}_n}\}$. In the case $ |\eta|_{\mathbf{H}_n}^2\geq t$, we find
\begin{align*}
\int_0^{\frac{t}{2}} & \int_{\{ |\zeta|_{\mathbf{H}_n}\leq 2^{-1} |\eta|_{\mathbf{H}_n}\} } h_{t-s}(\zeta)\left(1+s+|\eta\circ\zeta^{-1}|^2_{\mathbf{H}_n}\right)^{-\alpha} \, d\zeta \, ds \\
& \quad \lesssim \int_0^{\frac{t}{2}} \left(1+s+|\eta|^2_{\mathbf{H}_n}\right)^{-\alpha} \int_{\{ |\zeta|_{\mathbf{H}_n}\leq 2^{-1} |\eta|_{\mathbf{H}_n}\} } h_{t-s}(\zeta) \, d\zeta \, ds \\
& \quad \lesssim \left(1+|\eta|^2_{\mathbf{H}_n}\right)^{-\alpha} \int_0^{\frac{t}{2}}  \int_{\mathbf{H}_n} h_{t-s}(\zeta) \, d\zeta \, ds \lesssim t \left(1+|\eta|^2_{\mathbf{H}_n}\right)^{-\alpha} \simeq t \left(1+t+|\eta|^2_{\mathbf{H}_n}\right)^{-\alpha} .
\end{align*} Otherwise, if $t\geq |\eta|_{\mathbf{H}_n}^2$, then,
\begin{align*}
\int_0^{\frac{t}{2}} & \int_{\{ |\zeta|_{\mathbf{H}_n}\leq 2^{-1} |\eta|_{\mathbf{H}_n}\} } h_{t-s}(\zeta)\left(1+s+|\eta\circ\zeta^{-1}|^2_{\mathbf{H}_n}\right)^{-\alpha} \, d\zeta \, ds \\
& \quad \lesssim \int_0^{\frac{t}{2}} \int_{\{ |\zeta|_{\mathbf{H}_n}\leq 2^{-1} |\eta|_{\mathbf{H}_n}\} } (t-s)^{-\frac{Q}{2}}e^{-\frac{c}{(t-s)}|\zeta|^2_{\mathbf{H}_n}}   \left(1+s+|\zeta|^2_{\mathbf{H}_n}\right)^{-\alpha}\, d\zeta \, ds \\
& \quad \lesssim t^{-\frac{Q}{2}} \int_0^{\frac{t}{2}} \int_{\{ |\zeta|_{\mathbf{H}_n}\leq 2^{-1} |\eta|_{\mathbf{H}_n}\} }   \left(1+s+|\zeta|^2_{\mathbf{H}_n}\right)^{-\alpha}\, d\zeta \, ds \\
& \quad \lesssim  t^{-\frac{Q}{2}} \int_0^{\frac{t}{2}} \int_0^{2^{-1} |\eta|_{\mathbf{H}_n}} \iint_{ |(x',y')|\leq 2^{-1} |\eta|_{\mathbf{H}_n} }   \left(1+s+|x'|^2+|y'|^2+\tau'\right)^{-\alpha}\, d(x',y')\, d\tau' \, ds .
\end{align*} Carrying out the change of variables $s=\sigma^2$ and $\tau'=\omega^2$, from the last estimate in the case $\alpha\in\left(0,1+\frac{Q}{2}\right)$ we find
\begin{align*}
\int_0^{\frac{t}{2}} & \int_{\{ |\zeta|_{\mathbf{H}_n}\leq 2^{-1} |\eta|_{\mathbf{H}_n}\} } h_{t-s}(\zeta)\left(1+s+|\eta\circ\zeta^{-1}|^2_{\mathbf{H}_n}\right)^{-\alpha} \, d\zeta \, ds \\
& \quad \lesssim t^{-\frac{Q}{2}} \iiiint_{ |(x',y',\omega,\sigma)|^2\lesssim \, t+|\eta|_{\mathbf{H}_n}^2 }   \left(1+|\sigma|^2+|x'|^2+|y'|^2+|\omega|^2\right)^{-\alpha} \sigma\omega \, d(x',y',\omega,\sigma) \\
& \quad \lesssim t^{-\frac{Q}{2}} \iiiint_{ |(x',y',\omega,\sigma)|^2\lesssim \, t+|\eta|_{\mathbf{H}_n}^2 }   \left(1+|\sigma|^2+|x'|^2+|y'|^2+|\omega|^2\right)^{-\alpha+1} \, d(x',y',\omega,\sigma)  \\
& \quad \lesssim t^{-\frac{Q}{2}} \int_{0<\, \varrho\,  \lesssim \, (t+|\eta|_{\mathbf{H}_n}^2)^{1/2} }   \left(1+\varrho^2\right)^{-\alpha+1} \varrho^{Q-1}\, d\varrho \lesssim t^{-\frac{Q}{2}} \int_{1<\, \varrho\,  \lesssim \, (1+t+|\eta|_{\mathbf{H}_n}^2)^{1/2} }   \varrho^{-2\alpha+Q+1}\, d\varrho \\
 & \quad \lesssim t^{-\frac{Q}{2}} (1+t+|\eta|_{\mathbf{H}_n}^2)^{-\alpha+\frac{Q}{2}+1} \simeq t \, (1+t+|\eta|_{\mathbf{H}_n}^2)^{-\alpha},
\end{align*} where we used the condition $-2\alpha+Q+1>-1$ and the equivalence $1+t+|\eta|_{\mathbf{H}_n}^2\simeq t$ for $|\eta|^2_{\mathbf{H}_n}\leq t$.  In the limit case $\alpha=1+Q/2$, for $|\eta|^2_{\mathbf{H}_n}\leq t$ we have to include a logarithmic term in the previous estimate, namely,
\begin{align*}
\int_0^{\frac{t}{2}} \int_{\{ |\zeta|_{\mathbf{H}_n}\leq 2^{-1} |\eta|_{\mathbf{H}_n}\} } h_{t-s}(\zeta)\left(1+s+|\eta\circ\zeta^{-1}|^2_{\mathbf{H}_n}\right)^{-\alpha}\, d\zeta \, ds  & \lesssim t^{-\frac{Q}{2}} \log (e+t) =  t^{1-\alpha} \log (e+t) \\ & \simeq  t \, (1+t+|\eta|_{\mathbf{H}_n}^2)^{-\alpha} \log (e+t)  .
\end{align*}
 We proceed now with the estimate of the integral in the intermediate sub-region  $\{\zeta\in \mathbf{H}_n:  2^{-1}  |\eta|_{\mathbf{H}_n} \leq |\zeta|_{\mathbf{H}_n}\leq 2  |\eta|_{\mathbf{H}_n}\}$. Since in this region for $s\in [0,\frac{t}{2}]$ we may estimate $(t-s)^{-\frac{Q}{2}}\simeq t^{-\frac{Q}{2}}$ and $e^{-\frac{c}{t-s}|\zeta|_{\mathbf{H}_n}^2}\leq e^{-\frac{c}{4t}|\eta|_{\mathbf{H}_n}^2}$, then, it results
\begin{align*}
& \int_0^{\frac{t}{2}}  \int_{\{  2^{-1}  |\eta|_{\mathbf{H}_n} \leq |\zeta|_{\mathbf{H}_n}\leq 2  |\eta|_{\mathbf{H}_n}\} } h_{t-s}(\zeta)\left(1+s+|\eta\circ\zeta^{-1}|^2_{\mathbf{H}_n}\right)^{-\alpha} \, d\zeta \, ds \\
& \qquad \lesssim \int_0^{\frac{t}{2}}  \int_{\{  2^{-1}  |\eta|_{\mathbf{H}_n} \leq |\zeta|_{\mathbf{H}_n}\leq 2  |\eta|_{\mathbf{H}_n}\} } (t-s)^{-\frac{Q}{2}} e^{-\frac{c}{t-s}|\zeta|_{\mathbf{H}_n}^2} \left(1+s+|\eta\circ\zeta^{-1}|^2_{\mathbf{H}_n}\right)^{-\alpha} \, d\zeta \, ds \\
& \qquad \lesssim t^{-\frac{Q}{2}} e^{-\frac{c}{4t}|\eta|_{\mathbf{H}_n}^2}  \int_0^{\frac{t}{2}}  \int_{\{  2^{-1}  |\eta|_{\mathbf{H}_n} \leq |\zeta|_{\mathbf{H}_n}\leq 2  |\eta|_{\mathbf{H}_n}\} }  \left(1+s+|\eta\circ\zeta^{-1}|^2_{\mathbf{H}_n}\right)^{-\alpha} \, d\zeta \, ds  \\
& \qquad \lesssim t^{-\frac{Q}{2}} e^{-\frac{c}{4t}|\eta|_{\mathbf{H}_n}^2}  \int_0^{\frac{t}{2}}  \int_{\{   |\xi|_{\mathbf{H}_n}\leq 3  |\eta|_{\mathbf{H}_n}\} }  \left(1+s+|\xi|^2_{\mathbf{H}_n}\right)^{-\alpha} \, d\xi \, ds ,
\end{align*} where we employed the change of variables $\xi= \zeta\circ \eta^{-1}$ in the last step (note that  $2^{-1}  |\eta|_{\mathbf{H}_n} \leq |\zeta|_{\mathbf{H}_n}\leq 2  |\eta|_{\mathbf{H}_n}$ implies $ |\xi|_{\mathbf{H}_n}\leq 3  |\eta|_{\mathbf{H}_n}$ and $d\xi=d\zeta$). If we denote $\xi=(x',y'\tau')$, then, since $|\xi|_{\mathbf{H}_n}\simeq |(x',y')|+|\tau'|^{\frac{1}{2}}$ we have
\begin{align*}
& \int_0^{\frac{t}{2}}  \int_{\{  2^{-1}  |\eta|_{\mathbf{H}_n} \leq |\zeta|_{\mathbf{H}_n}\leq 2  |\eta|_{\mathbf{H}_n}\} } h_{t-s}(\zeta)\left(1+s+|\eta\circ\zeta^{-1}|^2_{\mathbf{H}_n}\right)^{-\alpha} \, d\zeta \, ds \\
& \qquad \lesssim t^{-\frac{Q}{2}} e^{-\frac{c}{4t}|\eta|_{\mathbf{H}_n}^2}  \int_0^{\frac{t}{2}} \int_{|\tau'|\lesssim |\eta|_{\mathbf{H}_n}^2} \iint_{  |(x',y')| \lesssim  |\eta|_{\mathbf{H}_n} }  \left(1+s+|\tau'|+|x'|^2+|y'|^2\right)^{-\alpha}\, d(x',y') \, d\tau' \, ds \\
& \qquad \lesssim t^{-\frac{Q}{2}} e^{-\frac{c}{4t}|\eta|_{\mathbf{H}_n}^2}  \int_0^{t^{1/2}} \int_{|\omega|\lesssim |\eta|_{\mathbf{H}_n}} \iint_{  |(x',y')| \lesssim  |\eta|_{\mathbf{H}_n} }  \left(1+\sigma^2+\omega^2+|x'|^2+|y'|^2\right)^{-\alpha} \sigma \omega\, d(x',y') \, d\omega \, d\sigma \\
& \qquad \lesssim t^{-\frac{Q}{2}} e^{-\frac{c}{4t}|\eta|_{\mathbf{H}_n}^2} \iiiint_{  |(x',y',\omega,\sigma)| \lesssim (t+ |\eta|_{\mathbf{H}_n}^2)^{1/2} }  \left(1+\sigma^2+\omega^2+|x'|^2+|y'|^2\right)^{-\alpha+1}\, d(x',y') \, d\omega \, d\sigma \\
& \qquad \lesssim t^{-\frac{Q}{2}} e^{-\frac{c}{4t}|\eta|_{\mathbf{H}_n}^2} \int_{ 0<\varrho\lesssim  (t+ |\eta|_{\mathbf{H}_n}^2)^{1/2} }  \left(1+\varrho^2\right)^{-\alpha+1} \varrho^{Q-1}\, d\varrho \lesssim t^{-\frac{Q}{2}} e^{-\frac{c}{4t}|\eta|_{\mathbf{H}_n}^2} \int_{ 1<\varrho\lesssim  (1+t+ |\eta|_{\mathbf{H}_n}^2)^{1/2} } \varrho^{-2\alpha+Q+1}\, d\varrho,
\end{align*} where we carried out the change of variables $s=\sigma^2$ and $\tau'=\omega^2$ in the second inequality and we reduce the resulting integral to the integral of a radial symmetric function with $2n+2$ variables. As we have already noticed, for $\alpha\in \left(0,1+Q/2\right)$ the power of the integrand in the last integral is greater than $-1$, therefore, we obtain
\begin{align*}
& \int_0^{\frac{t}{2}}  \int_{\{  2^{-1}  |\eta|_{\mathbf{H}_n} \leq |\zeta|_{\mathbf{H}_n}\leq 2  |\eta|_{\mathbf{H}_n}\} } h_{t-s}(\zeta)\left(1+s+|\eta\circ\zeta^{-1}|^2_{\mathbf{H}_n}\right)^{-\alpha} \, d\zeta \, ds \\
& \qquad  \lesssim  t^{-\frac{Q}{2}} e^{-\frac{c}{4t}|\eta|_{\mathbf{H}_n}^2} \left(1+t+ |\eta|_{\mathbf{H}_n}^2\right)^{-\alpha+\frac{Q}{2}+1} \simeq t \left(1+t+ |\eta|_{\mathbf{H}_n}^2\right)^{-\alpha}  t^{-\frac{Q}{2}-1}  \left(t+ |\eta|_{\mathbf{H}_n}^2\right)^{\frac{Q}{2}+1} e^{-\frac{c}{4t}|\eta|_{\mathbf{H}_n}^2} \\
& \qquad   \lesssim t \left(1+t+ |\eta|_{\mathbf{H}_n}^2\right)^{-\alpha}.
\end{align*}
 On the other hand, in the limit case $\alpha=1+Q/2$, we get an extra logarithmic terms, namely,
\begin{align*}
 \int_0^{\frac{t}{2}} &  \int_{\{  2^{-1}  |\eta|_{\mathbf{H}_n} \leq |\zeta|_{\mathbf{H}_n}\leq 2  |\eta|_{\mathbf{H}_n}\} } h_{t-s}(\zeta)\left(1+s+|\eta\circ\zeta^{-1}|^2_{\mathbf{H}_n}\right)^{-\alpha} \, d\zeta \, ds \\
& \qquad  \lesssim t^{-\frac{Q}{2}} e^{-\frac{c}{4t}|\eta|_{\mathbf{H}_n}^2} \int_{ 1<\varrho \lesssim  (1+t+ |\eta|_{\mathbf{H}_n}^2)^{1/2} } \varrho^{-1}\, d\varrho  \simeq t^{-\frac{Q}{2}} e^{-\frac{c}{4t}|\eta|_{\mathbf{H}_n}^2} \log\left(1+t+ |\eta|_{\mathbf{H}_n}^2\right) \\
& \qquad \lesssim t \left(1+t+ |\eta|_{\mathbf{H}_n}^2\right)^{-\frac{Q}{2}-1}   \left(1+\tfrac{ |\eta|_{\mathbf{H}_n}^2}{t}\right)^{\frac{Q}{2}+1} e^{-\frac{c}{4t}|\eta|_{\mathbf{H}_n}^2} \log\left(2\left(t+ |\eta|_{\mathbf{H}_n}^2\right)\right) \\
& \qquad \simeq t \left(1+t+ |\eta|_{\mathbf{H}_n}^2\right)^{-\alpha}   \left(1+\tfrac{ |\eta|_{\mathbf{H}_n}^2}{t}\right)^{\frac{Q}{2}+1} e^{-\frac{c}{4t}|\eta|_{\mathbf{H}_n}^2} \left(\log\left(2\left(1+ \tfrac{|\eta|_{\mathbf{H}_n}^2}{t}\right)\right)+\log t\right) \\
& \qquad \lesssim t \left(1+t+ |\eta|_{\mathbf{H}_n}^2\right)^{-\alpha}  \log (e+t).
\end{align*} Finally, we estimate the integral in the sub-region
$\{\zeta\in \mathbf{H}_n:  |\zeta|_{\mathbf{H}_n}\geq 2  |\eta|_{\mathbf{H}_n}\}$. We have
\begin{align*}
& \int_0^{\frac{t}{2}}  \int_{\{   |\zeta|_{\mathbf{H}_n}\geq 2  |\eta|_{\mathbf{H}_n}\} } h_{t-s}(\zeta)\left(1+s+|\eta\circ\zeta^{-1}|^2_{\mathbf{H}_n}\right)^{-\alpha} \, d\zeta \, ds  \\
& \qquad \lesssim  \int_0^{\frac{t}{2}}  \int_{\{   |\zeta|_{\mathbf{H}_n}\geq 2  |\eta|_{\mathbf{H}_n}\} }(t-s)^{-\frac{Q}{2}} e^{-\frac{c}{t-s}|\zeta|^2_{\mathbf{H}_n}}\left(1+s+|\eta\circ\zeta^{-1}|^2_{\mathbf{H}_n}\right)^{-\alpha} \, d\zeta \, ds \\
& \qquad \lesssim t^{-\frac{Q}{2}} \int_0^{\frac{t}{2}}  \int_{\{   |\zeta|_{\mathbf{H}_n}\geq 2  |\eta|_{\mathbf{H}_n}\} } e^{-\frac{c}{t}|\zeta|^2_{\mathbf{H}_n}}\left(1+s+|\zeta|^2_{\mathbf{H}_n}\right)^{-\alpha} \, d\zeta \, ds \\
& \qquad \simeq t^{-\frac{Q}{2}} \int_0^{\frac{t}{2}}  \iiint_{ |(x',y',\tau')|_{\mathbf{H}_n}\geq 2  |\eta|_{\mathbf{H}_n} } e^{-\frac{c}{t}(s+|(x',y',\tau')|^2_{\mathbf{H}_n})}\left(1+s+|\tau'|+|x'|^2+|y'|^2\right)^{-\alpha} \, d(x',y',\tau') \, ds  \\
& \qquad \lesssim t^{-\frac{Q}{2}} \int_0^{t^{1/2}}  \iiint_{   |(x',y',\omega^2)|_{\mathbf{H}_n}\geq 2  |\eta|_{\mathbf{H}_n} } e^{-\frac{c}{t}(\sigma^2+|(x',y',\omega^2)|^2_{\mathbf{H}_n})}\left(1+\sigma^2+|\omega|^2+|x'|^2+|y'|^2\right)^{-\alpha} \sigma\omega\, d(x',y',\omega
) \, d\sigma \\
& \qquad \lesssim t^{-\frac{Q}{2}} \iiiint_{  |(x',y',\omega,\sigma)|\gtrsim  |\eta|_{\mathbf{H}_n} } e^{-\frac{c}{2t}(\sigma^2+|x'|^2+|y'|^2+|\omega|^2)}\left(1+\sigma^2+|\omega|^2+|x'|^2+|y'|^2\right)^{-\alpha+1} \, d(x',y',\omega,\sigma),
\end{align*} where we applied the usual change of variables $s=\sigma^2$ and $\tau'=\omega^2$. As in the last integral the function is radial symmetric, we get
\begin{align}
& \int_0^{\frac{t}{2}}  \int_{\{   |\zeta|_{\mathbf{H}_n}\geq 2  |\eta|_{\mathbf{H}_n}\} } h_{t-s}(\zeta)\left(1+s+|\eta\circ\zeta^{-1}|^2_{\mathbf{H}_n}\right)^{-\alpha} \, d\zeta \, ds \notag \\
& \qquad \lesssim t^{-\frac{Q}{2}} \int_{ \varrho \gtrsim  |\eta|_{\mathbf{H}_n} } e^{-\frac{c}{2t}\varrho^2}\left(1+\varrho^2\right)^{-\alpha+1} \varrho^{Q-1}\, d\varrho\lesssim t^{-\frac{Q}{2}} \int_{  \varrho\gtrsim  |\eta|_{\mathbf{H}_n} } e^{-\frac{c}{2t}\varrho^2}\varrho^{-2\alpha+Q+1}\, d\varrho \notag \\
& \qquad \lesssim t^{1-\alpha} \int_{  \varrho\gtrsim \frac{ |\eta|_{\mathbf{H}_n}}{\sqrt{t}} } e^{-\frac{cR^2}{2}}R^{-2\alpha+Q+1}\, dR \lesssim t^{1-\alpha}  e^{-\frac{c'}{t}|\eta|^2_{\mathbf{H}_n}} \int_{  \varrho\gtrsim \frac{ |\eta|_{\mathbf{H}_n}}{\sqrt{t}} } e^{-\frac{cR^2}{4}}R^{-2\alpha+Q+1}\, dR. \label{estimate I2 large zeta}
\end{align}
 For $0<\alpha<1+Q/2$ since the power $-2\alpha+Q+1$ is strictly greater than $-1$, we have that $e^{-\frac{cR}{4}}R^{-2\alpha+Q+1}\in L^1(\mathbb{R}_+)$ and, consequently, \eqref{estimate I2 large zeta} implies
\begin{align*}
 \int_0^{\frac{t}{2}}  \int_{\{   |\zeta|_{\mathbf{H}_n}\geq 2  |\eta|_{\mathbf{H}_n}\} } h_{t-s}(\zeta)\left(1+s+|\eta\circ\zeta^{-1}|^2_{\mathbf{H}_n}\right)^{-\alpha} \, d\zeta \, ds & \lesssim t^{1-\alpha}  e^{-\frac{c'}{t}|\eta|^2_{\mathbf{H}_n}} \\ & \lesssim t \left(1+t+|\eta|^2_{\mathbf{H}_n}\right)^{-\alpha} \left(1+\tfrac{|\eta|^2_{\mathbf{H}_n}}{t}\right)^{\alpha} e^{-\frac{c'}{t}|\eta|^2_{\mathbf{H}_n}}  \\ &\lesssim t \left(1+t+|\eta|^2_{\mathbf{H}_n}\right)^{-\alpha}.
\end{align*}  In the limit case $\alpha=1+Q/2$, we distinguish two subcases. When $|\eta|^2_{\mathbf{H}_n}\geq t$, since in the integral in the right-hand side of \eqref{estimate I2 large zeta} we are away from 0, it follows that $e^{-\frac{cR}{4}}R^{-1}$ is summable. Therefore, we may repeat exactly the same estimates as in the previous case. On the other hand, if $|\eta|^2_{\mathbf{H}_n}\leq t$, we need to modify slightly \eqref{estimate I2 large zeta} as follows:
\begin{align}
& \int_0^{\frac{t}{2}}  \int_{\{   |\zeta|_{\mathbf{H}_n}\geq 2  |\eta|_{\mathbf{H}_n}\} } h_{t-s}(\zeta)\left(1+s+|\eta\circ\zeta^{-1}|^2_{\mathbf{H}_n}\right)^{-\alpha} \, d\zeta \, ds \notag \\
& \qquad \lesssim t^{-\frac{Q}{2}} \int_{ \varrho \gtrsim  |\eta|_{\mathbf{H}_n} } e^{-\frac{c}{2t}\varrho^2}\left(1+\varrho^2\right)^{-\alpha+1} \varrho^{Q-1}\, d\varrho \lesssim t^{-\frac{Q}{2}}   e^{-\frac{c'}{t}|\eta|^2_{\mathbf{H}_n}} \int_{  \varrho\gtrsim  |\eta|_{\mathbf{H}_n} } e^{-\frac{c}{4t}(\varrho^2+1)}e^{\frac{c}{4t}}(1+\varrho)^{-2\alpha+Q+1}\, d\varrho \notag \\
& \qquad \lesssim t^{-\frac{Q}{2}}   e^{-\frac{c'}{t}|\eta|^2_{\mathbf{H}_n}} \int_{  \varrho\gtrsim  |\eta|_{\mathbf{H}_n} } e^{-\frac{c}{8t}(\varrho+1)^2}(1+\varrho)^{-1}\, d\varrho  \lesssim t^{-\frac{Q}{2}}   e^{-\frac{c'}{t}|\eta|^2_{\mathbf{H}_n}}  \int_{  R\gtrsim \frac{1+ |\eta|_{\mathbf{H}_n}}{\sqrt{t}} } e^{-\frac{cR^2}{8}}R^{-1}\, dR . \label{estimate I2 large zeta 2nd case}
\end{align} Note that in the previous chain of inequalities we used the fact that $e^{\frac{c}{4t}}$ is a bounded function in the case $|\eta|^2_{\mathbf{H}_n}\leq t$. Indeed, in the case that we are considering it holds $t\gtrsim 1$ (keep in mind that we are in the case $t+|\eta|_{\mathbf{H}_n}^2\gtrsim 1$). Once again, if the lower bound of the domain of integration in the last integral is grater than 1, we get
\begin{align*}
\int_0^{\frac{t}{2}}  \int_{\{   |\zeta|_{\mathbf{H}_n}\geq 2  |\eta|_{\mathbf{H}_n}\} } h_{t-s}(\zeta)\left(1+s+|\eta\circ\zeta^{-1}|^2_{\mathbf{H}_n}\right)^{-\alpha} \, d\zeta \, ds &\lesssim  t^{-\frac{Q}{2}}   e^{-\frac{c'}{t}|\eta|^2_{\mathbf{H}_n}} \\ &\lesssim  t \left(1+t+|\eta|^2_{\mathbf{H}_n}\right)^{-\frac{Q}{2}-1}=  t \left(1+t+|\eta|^2_{\mathbf{H}_n}\right)^{-\alpha},
\end{align*} otherwise,
\begin{align*}
 \int_{  R\gtrsim \frac{1+ |\eta|_{\mathbf{H}_n}}{\sqrt{t}} } e^{-\frac{cR^2}{8}}R^{-1}\, dR & =   \int_{ 1}^\infty e^{-\frac{cR^2}{8}}R^{-1}\, dR +  \int_{   \frac{1+ |\eta|_{\mathbf{H}_n}}{\sqrt{t}} \lesssim R <1} e^{-\frac{cR^2}{8}}R^{-1}\, dR   \lesssim  1+ \int_{   \frac{1+ |\eta|_{\mathbf{H}_n}}{\sqrt{t}} \lesssim R <1} R^{-1}\, dR  \\
 & = 1+\log \left(\tfrac{\sqrt{t}}{1+ |\eta|_{\mathbf{H}_n}}\right) \lesssim 1+ \log t \lesssim \log (e+t)
\end{align*} and, consequently,
\begin{align*}
& \int_0^{\frac{t}{2}}  \int_{\{   |\zeta|_{\mathbf{H}_n}\geq 2  |\eta|_{\mathbf{H}_n}\} } h_{t-s}(\zeta)\left(1+s+|\eta\circ\zeta^{-1}|^2_{\mathbf{H}_n}\right)^{-\alpha} \, d\zeta \, ds  \\
& \qquad  \lesssim  t^{-\frac{Q}{2}}   e^{-\frac{c'}{t}|\eta|^2_{\mathbf{H}_n}}  \log (e+t)  \lesssim  t \left(1+t+|\eta|^2_{\mathbf{H}_n}\right)^{-\frac{Q}{2}-1} \log (e+t)=  t \left(1+t+|\eta|^2_{\mathbf{H}_n}\right)^{-\alpha} \log (e+t).
\end{align*}
Combining all possible subcases, we proved eventually \eqref{fundamental estimate Duhamel term inhom}.
\end{proof}

\subsection{A priori estimates}

Combining the results from Propositions \ref{Prop fund estimate linear term inhom case}, \ref{Prop fund estimate linear term inhom case kappa<Q}  and \ref{Prop fund estimate Duhamel term inhom}, we can prove now the following a priori estimates. These  will play a fundamental role in the proof of Theorems \ref{Thm GESD} and \ref{Thm LER}.

\begin{proposition} Let $\kappa$ be a positive parameter and $T\in (0,\infty]$. Moreover, we consider $u_0\in (1+|\cdot|_{\mathbf{H}_n}^2)^{-\frac{\kappa}{2}}L^\infty (\mathbf{H}_n)$ and  a source term $F: [0,T)\times \mathbf{H}_n\to \mathbb{R}$.
\begin{enumerate}
\item If $\kappa\neq Q$, then,
\begin{align}\label{a priori est hom sol}
|e^{t\Delta_{\hor}}u_0(\eta)| \leq C_0  \big(1+t+|\eta|_{\mathbf{H}_n}^2\big)^{-\frac{1}{2}\min\{\kappa,Q\}}\big\| \big(1+|\eta|^2_{\mathbf{H}_n}\big)^{\frac{\kappa}{2}} u_0(\eta)\big\|_{ L^\infty (\mathbf{H}_n)} 
\end{align}  for any $t\in [0,T)$ and $\eta\in \mathbf{H}_n$.
\item If $\kappa<Q$ and $\big(1+t+|\eta|^2_{\mathbf{H}_n}\big)^{\frac{\kappa}{2}+1} F\in  L^\infty\big(0,T\, ; L^\infty(\mathbf{H}_n) \big)$, then,
\begin{align} \label{a priori est Duh sol}
\Big| \int_0^t e^{(t-s)\Delta_{\hor}}F(s) \, ds (\eta)\Big|\leq C_1 \big(1+t+|\eta|_{\mathbf{H}_n}^2\big)^{-\frac{\kappa}{2}} \big\| \big(1+t+|\eta|_{\mathbf{H}_n}^2\big)^{\frac{\kappa}{2}+1} F(t,\eta)\big\|_{L^{\infty}(0,T \, ; L^\infty (\mathbf{H}_n))}
\end{align} for any $t\in [0,T)$ and $\eta\in \mathbf{H}_n$.
\item  If $\theta\in [0,1)$ and $\big(1+t+|\eta|^2_{\mathbf{H}_n}\big)^{\frac{Q}{2}+\theta} F\in  L^\infty\big(0,T \, ; L^\infty(\mathbf{H}_n) \big)$ and $T<\infty$, then,
\begin{align} \label{a priori est Duh sol theta}
\Big| \int_0^t e^{(t-s)\Delta_{\hor}}F(s)\, ds (\eta)\Big| \leq C_\theta T^{1-\theta} \big(1+t+|\eta|_{\mathbf{H}_n}^2\big)^{-\frac{Q}{2}} \big\| \big(1+t+|\eta|_{\mathbf{H}_n}^2\big)^{\frac{Q}{2}+\theta} F(t,\eta)\big\|_{L^{\infty}(0,T \, ; L^\infty (\mathbf{H}_n))}
\end{align} for any $t\in [0,T)$ and $\eta\in \mathbf{H}_n$.
\item  If $\big(1+t+|\eta|^2_{\mathbf{H}_n}\big)^{\frac{Q}{2}+1} F\in  L^\infty\big(0,T \, ; L^\infty(\mathbf{H}_n) \big)$ and $T<\infty$, then,
\begin{align} \label{a priori est Duh sol log}
\Big| \int_0^t e^{(t-s)\Delta_{\hor}}F(s)\, ds (\eta)\Big| \leq C_2 \log(e+T) \big(1+t+|\eta|_{\mathbf{H}_n}^2\big)^{-\frac{Q}{2}} \big\| \big(1+t+|\eta|_{\mathbf{H}_n}^2\big)^{\frac{Q}{2}+1} F(t,\eta)\big\|_{L^{\infty}(0,T \, ; L^\infty (\mathbf{H}_n))}
\end{align} for any $t\in [0,T)$ and $\eta\in \mathbf{H}_n$.
\end{enumerate}
Here $C_0,C_1,C_\theta$ and $C_2$ denote positive constants independent of $T$.
\end{proposition}

\begin{proof}
Let us begin with the estimate of the solution of the homogeneous problem. Since $0<\kappa\neq Q$, by using \eqref{fundamental estimate linear term inhom case} and \eqref{fundamental estimate linear term inhom case kappa<Q}, we get
\begin{align*}
\left|e^{t\Delta_{\hor}}u_0(\eta)\right| & \leq   \int_{\mathbf{H}_n}h_t(\zeta^{-1}\circ \eta) |u_0(\zeta)| \, d\zeta  \leq \big\| \big(1+ |\cdot|^2_{\mathbf{H}_n}\big)^{\frac{\kappa}{2}} u_0\big\|_{ L^\infty (\mathbf{H}_n)}  \int_{\mathbf{H}_n}h_t(\zeta^{-1}\circ \eta)\big(1+ |\zeta|^{2}_{\mathbf{H}_n}\big)^{-\frac{\kappa}{2}} \, d\zeta \\
& =\big\| \big(1+ |\cdot|^2_{\mathbf{H}_n}\big)^{\frac{\kappa}{2}} u_0\big\|_{ L^\infty (\mathbf{H}_n)}  \Big(e^{t\Delta_{\hor}}\big(1+|\cdot|^{2}_{\mathbf{H}_n}\big)^{-\frac{\kappa}{2}}\Big)(\eta)\\ & \lesssim   \big(1+t+|\eta|_{\mathbf{H}_n}^2\big)^{-\frac{1}{2}\min\{\kappa,Q\}} \big\| \big(1+ |\cdot|^2_{\mathbf{H}_n}\big)^{\frac{\kappa}{2}} u_0\big\|_{ L^\infty (\mathbf{H}_n)} .
\end{align*} We prove now the second estimate. Applying \eqref{fundamental estimate Duhamel term inhom} for $\alpha=\frac{\kappa}{2}+1<1+\frac{Q}{2}$, it results
\begin{align*}
 \left|\int_0^t e^{(t-s)\Delta_{\hor}}F(s)\, ds (\eta)\right|
& \leq \int_0^t \int_{\mathbf{H}_n} h_{t-s}(\zeta^{-1}\circ \eta)|F(s,\zeta)|\, d\zeta \, ds \\
 & \leq \big\| \big(1+t+|\eta|_{\mathbf{H}_n}^2\big)^{\frac{\kappa}{2}+1} F(t,\eta)\big\|_{L^{\infty}(0,T \, ; L^\infty (\mathbf{H}_n))} \int_0^t \int_{\mathbf{H}_n} h_{t-s}(\zeta^{-1}\circ \eta)\big(1+s+|\zeta|_{\mathbf{H}_n}^2\big)^{-\frac{\kappa}{2}-1}\, d\zeta \, ds \\
 & \lesssim  t \, \big(1+t+|\eta|_{\mathbf{H}_n}^2\big)^{-\frac{\kappa}{2}-1}\| \big(1+t+|\eta|_{\mathbf{H}_n}^2\big)^{\frac{\kappa}{2}+1} F(t,\eta)\big\|_{L^{\infty}(0,T \, ; L^\infty (\mathbf{H}_n))}   \\
 & \lesssim \big(1+t+ |\eta|_{\mathbf{H}_n}^2\big)^{-\frac{\kappa}{2}}  \| \big(1+t+|\eta|_{\mathbf{H}_n}^2\big)^{\frac{\kappa}{2}+1} F(t,\eta)\big\|_{L^{\infty}(0,T \, ; L^\infty (\mathbf{H}_n))}.
\end{align*} Similarly, for $\alpha=\frac{Q}{2}+\theta<1+\frac{Q}{2}$, \eqref{fundamental estimate Duhamel term inhom} yields
\begin{align*}
\left|\int_0^t e^{(t-s)\Delta_{\hor}}F(s)\, ds (\eta)\right|
& \leq \int_0^t \int_{\mathbf{H}_n} h_{t-s}(\zeta^{-1}\circ \eta)|F(s,\zeta)|\, d\zeta \, ds \\
 & \leq \big\| \big(1+t+|\eta|_{\mathbf{H}_n}^2\big)^{\frac{Q}{2}+\theta} F(t,\eta)\big\|_{L^{\infty}(0,T \, ; L^\infty (\mathbf{H}_n))} \int_0^t \int_{\mathbf{H}_n} h_{t-s}(\zeta^{-1}\circ \eta)\big(1+s+|\zeta|_{\mathbf{H}_n}^2\big)^{-\frac{Q}{2}-\theta}\, d\zeta \, ds \\
 & \lesssim  t \, \big(1+t+|\eta|_{\mathbf{H}_n}^2\big)^{-\frac{Q}{2}-\theta}\| \big(1+t+|\eta|_{\mathbf{H}_n}^2\big)^{\frac{Q}{2}+\theta} F(t,\eta)\big\|_{L^{\infty}(0,T \, ; L^\infty (\mathbf{H}_n))}   \\
 & \lesssim T^{1-\theta}\big(1+t+ |\eta|_{\mathbf{H}_n}^2\big)^{-\frac{Q}{2}}  \| \big(1+t+|\eta|_{\mathbf{H}_n}^2\big)^{\frac{Q}{2}+\theta} F(t,\eta)\big\|_{L^{\infty}(0,T \, ; L^\infty (\mathbf{H}_n))}.
\end{align*}
 So, we proved \eqref{a priori est Duh sol theta} as well. Finally, the proof of \eqref{a priori est Duh sol log} is completely analogous to that of \eqref{a priori est Duh sol theta} (formally for $\theta=1$), the only difference is that we have to employ \eqref{fundamental estimate Duhamel term inhom} in the limit case $\alpha=1+Q/2$, obtaining in this way the logarithmic factor.  The proof is complete.
\end{proof}

\section{Global existence of small data solutions in the supercritical case}

In this and in next section, we will prove the global in time existence of small data solutions in the super-Fuijita case and the local in time existence of solutions in the sub-Fujita case for the Cauchy problem \eqref{semilinear heat Heisenberg}, respectively. As \emph{mild solutions} to \eqref{semilinear heat Heisenberg} we consider the solutions in certain weighted $L^\infty$ spaces of the nonlinear integral equation \eqref{nonlinear integral equation} as in Definition \ref{Def mild sol}.
 For this reason, we introduce the nonlinear integral operator
\begin{align*}
\Phi [u](t,\eta) & \doteq \varepsilon \, \big(e^{t\Delta_{\hor}} u_0\big)(\eta)+\int_0^t e^{(t-s)\Delta_{\hor}}|u(s)|^p \, ds (\eta) \\
& = \varepsilon \int_{\mathbf{H}_n} h_t(\zeta^{-1}\circ \eta)u_0(\zeta) \, d\zeta +\int_0^t \int_{\mathbf{H}_n} h_{t-s}(\zeta^{-1}\circ \eta)|u(s,\zeta)|^p \, d\zeta \, ds.
\end{align*} Therefore, our problem is reduced to find fixed points for the operator $\Phi$ in suitable function spaces.


\begin{proof}[Proof of Theorem \ref{Thm GESD}]
Let us denote $R_0\doteq \| \, (1+|\cdot|_{\mathbf{H}_n}^2)^{\frac{\kappa}{2}} u_0\|_{L^\infty(\mathbf{H}_n)}$. We shall prove that $\Phi$ is a contraction mapping from $\mathfrak{B}(R\varepsilon)\doteq \{u\in \mathrm{X}_{\kappa,T}: \|u\|_{\mathrm{X}_{\kappa,T}}\leq R\varepsilon\}$ into itself under suitable requirements for $R$ and $\varepsilon$. Combining \eqref{a priori est hom sol} and \eqref{a priori est Duh sol}, it follows
\begin{align*}
\| \Phi[u] \|_{X_{\kappa,T}}  & = \big\|\big(1+t+|\eta|_{\mathbf{H}_n}^2\big)^{\frac{\kappa}{2}}\Phi[u](t,\eta)\big\|_{L^\infty(0,T; L^\infty(\mathbf{H}_n))} \\
& \leq \varepsilon \big\|\big(1+t+|\eta|_{\mathbf{H}_n}^2\big)^{\frac{\kappa}{2}}e^{t\Delta_{\hor}}u_0\big\|_{L^\infty(0,T; L^\infty(\mathbf{H}_n))} + \Big\|\big(1+t+|\eta|_{\mathbf{H}_n}^2\big)^{\frac{\kappa}{2}}\int_0^t e^{(t-s)\Delta_{\hor}}|u(s,\cdot)|^p\, ds (\eta)\Big\|_{L^\infty(0,T; L^\infty(\mathbf{H}_n))} \\
& \leq  C_0\, \varepsilon \|(1+|\cdot|_{\mathbf{H}_n}^2)^{\frac{\kappa}{2}}u_0\|_{ L^\infty(\mathbf{H}_n)}+ C_1\big\|\big(1+t+|\eta|_{\mathbf{H}_n}^2\big)^{\frac{\kappa}{2}+1}|u(s,\eta)|^p\big\|_{L^\infty(0,T; L^\infty(\mathbf{H}_n))}\\
& = C_0\, \varepsilon \|(1+|\cdot|_{\mathbf{H}_n}^2)^{\frac{\kappa}{2}}u_0\|_{ L^\infty(\mathbf{H}_n)}+ C_1\big\|\big(1+t+|\eta|_{\mathbf{H}_n}^2\big)^{\frac{\kappa}{2}}u(s,\eta)\big\|_{L^\infty(0,T; L^\infty(\mathbf{H}_n))}^p \\
& = C_0 R_0\, \varepsilon + C_1\|u\|_{\mathrm{X}_{\kappa,T}}^p,
\end{align*} where in the second last step we employed the equality $\frac{\kappa}{2}+1=\frac{\kappa p}{2}$. Using again this relation for $\kappa$, \eqref{a priori est Duh sol} and the estimate
\begin{equation}\label{main value theorem |.|^p}
\big| |u|^p-|v|^p \big|\leq p |u-v|\big(|u|^{p-1}+|v|^{p-1}\big),
\end{equation} we arrive at
\begin{align*}
\| \Phi[u] -\Phi[v] \|_{X_{\kappa,T}} &= \Big\|\big(1+t+|\eta|_{\mathbf{H}_n}^2\big)^{\frac{\kappa}{2}}\int_0^t e^{(t-s)\Delta_{\hor}}\big(|u(s,\cdot)|^p-|v(s,\cdot)|^p\big)\, ds (\eta)\Big\|_{L^\infty(0,T; L^\infty(\mathbf{H}_n))} \\
& \leq   C_1\big\|\big(1+t+|\eta|_{\mathbf{H}_n}^2\big)^{\frac{\kappa}{2}+1}\big(|u(s,\eta)|^p-|v(s,\eta)|^p\big)\big\|_{L^\infty(0,T; L^\infty(\mathbf{H}_n))} \\
& =   p\, C_1\big\|\big(1+t+|\eta|_{\mathbf{H}_n}^2\big)^{\frac{\kappa}{2}p}|u(s,\eta)-v(s,\eta)|\big(|u(s,\eta)|^{p-1}+|v(s,\eta)|^{p-1}\big)\big\|_{L^\infty(0,T; L^\infty(\mathbf{H}_n))} \\
& \leq   p \, C_1\| u-v\|_{X_{\kappa,T}} \Big(\| u\|_{X_{\kappa,T}}^{p-1}+\| v\|_{X_{\kappa,T}}^{p-1}\Big).
\end{align*} Summarizing, we proved
\begin{align*}
\| \Phi[u] \|_{X_{\kappa,T}}  & \leq C_0 R_0\, \varepsilon + C_1\|u\|_{\mathrm{X}_{\kappa,T}}^p, \\
\| \Phi[u] -\Phi[v] \|_{X_{\kappa,T}} & \leq   p \, C_1\| u-v\|_{X_{\kappa,T}} \Big(\| u\|_{X_{\kappa,T}}^{p-1}+\| v\|_{X_{\kappa,T}}^{p-1}\Big).
\end{align*} Therefore, if we assume that
\begin{align*}
R= 2C_1R_0 \ \ \mbox{and} \ \ 0<\varepsilon\leq \varepsilon_0 \doteq \min\{(2C_1)^{-(p-1)}R^{-1},(4pC_1)^{-(p-1)}R^{-1}\},
\end{align*} then,
\begin{align*}
C_0R_0\, \varepsilon \leq  2^{-1}R\, \varepsilon,  \quad  C_1(R\varepsilon)^p \leq  2^{-1}R\, \varepsilon \quad \mbox{and} \quad 2pC_1 (R\varepsilon)^{p-1}\leq 2^{-1}.
\end{align*} In the above line, the first two relations imply that $\Phi$ maps $\mathfrak{B}(R\varepsilon)$ into itself, while the last inequality implies that $\Phi$ is a contraction with Lipschitz constant at most $2^{-1}$ on $\mathfrak{B}(R\varepsilon)$. Thus, by Banach-Caccioppoli fixed point theorem we have a unique mild solution $u\in \mathrm{X}_{\kappa,T}$. As the previous estimates are independent of $T$, we may prolong this unique solution for all times, obtaining a unique mild solution in $\mathrm{X}_{\kappa}$. Finally, the decay estimates follows immediately from $u\in \mathfrak{B}(R\varepsilon)$. This completes the proof.
\end{proof}

\section{Lower bound estimate for the lifespan  in the sub-Fujita case}

Next we prove a local in time existence result for mild solutions to the Cauchy problem in \eqref{semilinear heat Heisenberg} in the sub-Fujita case $p\in(1,p_{\Fuj}(Q)]$. Besides, we derive the lower bound estimate for the lifespan of the solution \eqref{lower bound lifespan}.


\begin{proof}[Proof of Theorem \ref{Thm LER}] Here, we will modify in a suitable way the proof of Theorem \ref{Thm GESD} in order to compensate the fact we are in the sub-Fujita case. Let us begin with the subcritical case. We define $\theta\doteq \frac{Q}{2}(p-1)$. Thanks to $1<p<p_{\Fuj}(Q)$ we get $\theta\in (0,1)$. Combining \eqref{a priori est hom sol} and \eqref{a priori est Duh sol theta}, we obtain
\begin{align*}
\| \Phi[u] \|_{X_{Q,T}}
& \leq  C_0\, \varepsilon \|(1+ |\cdot|_{\mathbf{H}_n}^2)^{\frac{\kappa}{2}}u_0\|_{ L^\infty(\mathbf{H}_n)}+   C_\theta T^{1-\theta} \big\|\big(1+t+|\eta|_{\mathbf{H}_n}^2\big)^{\frac{Q}{2}+\theta}|u(s,\eta)|^p\big\|_{L^\infty(0,T; L^\infty(\mathbf{H}_n))}\\
& = C_0\, \varepsilon \|(1+ |\cdot|_{\mathbf{H}_n}^2)^{\frac{\kappa}{2}}u_0\|_{ L^\infty(\mathbf{H}_n)}+   C_\theta T^{1-\theta}\big\|\big(1+t+|\eta|_{\mathbf{H}_n}^2\big)^{\frac{Q}{2}}u(s,\eta)\big\|_{L^\infty(0,T; L^\infty(\mathbf{H}_n))}^p \\
& = C_0 R_0\, \varepsilon +  C_\theta T^{1-\theta}\|u\|_{\mathrm{X}_{Q,T}}^p,
\end{align*}
where $R_0=\|(1+ |\cdot|_{\mathbf{H}_n}^2)^{\frac{\kappa}{2}}u_0\|_{L^\infty(\mathbf{H}_n)}$ as in the proof of Theorem \ref{Thm GESD} and in the second last step we used the equality $\frac{Q}{2}+\theta=\frac{Q p}{2}$. Analogously,
\begin{align*}
\| \Phi[u] -\Phi[v] \|_{X_{Q,T}} & \leq   C_\theta T^{1-\theta}\big\|\big(t+|\eta|_{\mathbf{H}_n}^2\big)^{\frac{Q}{2}+\theta}\big(|u(s,\eta)|^p-|v(s,\eta)|^p\big)\big\|_{L^\infty(0,T; L^\infty(\mathbf{H}_n))} \\
& \leq   p \,  C_\theta T^{1-\theta} \| u-v\|_{X_{Q,T}} \Big(\| u\|_{X_{Q,T}}^{p-1}+\| v\|_{X_{Q,T}}^{p-1}\Big),
\end{align*} where we used \eqref{main value theorem |.|^p} and the previous relation between $Q,p$ and $\theta$. Summarizing, we proved
\begin{align}
\| \Phi[u] \|_{X_{Q,T}}  & \leq C_0 R_0\, \varepsilon +  C_\theta T^{1-\theta}\|u\|_{\mathrm{X}_{Q,T}}^p , \label{1st ineq Phi} \\
\| \Phi[u] -\Phi[v] \|_{X_{Q,T}} & \leq p \,  C_\theta T^{1-\theta} \| u-v\|_{X_{Q,T}} \Big(\| u\|_{X_{Q,T}}^{p-1}+\| v\|_{X_{Q,T}}^{p-1}\Big) .\label{2nd ineq Phi}
\end{align} Therefore, if we require $R=2C_0 R_0$ and $T\leq C_{Q,p} \varepsilon^{-\frac{p-1}{1-\theta}}$, where $C_{Q,p}\doteq \min\left\{(2C_\theta)^{\frac{1}{\theta-1}} R^{\frac{1-p}{1-\theta}};(4pC_\theta)^{\frac{1}{\theta-1}} R^{\frac{1-p}{1-\theta}}\right\}$, then, from \eqref{1st ineq Phi} and \eqref{2nd ineq Phi} we get
\begin{equation} \label{Phi is a contraction}
\begin{split}
\| \Phi[u] \|_{X_{\kappa,T}}  & \leq R\varepsilon, \\
\| \Phi[u] -\Phi[v] \|_{X_{\kappa,T}} & \leq 2^{-1} \| u-v\|_{X_{\kappa,T}}
\end{split}
\end{equation} for any $u,v\in \mathfrak{B}(R\varepsilon)$.

Thus, we find a unique local in time solution to \eqref{semilinear heat Heisenberg} at least up to the time $C_{Q,p} \, \varepsilon^{-\frac{p-1}{1-\theta}}= C_{Q,p}\,  \varepsilon^{-(\frac{1}{p-1}-\frac{Q}{2})^{-1}}$. This implies immediately that the upper bound of the maximal time interval of existence for the local solution, that is, the lifespan $T_\varepsilon$, has to fulfill \eqref{lower bound lifespan} in the subcritical case.

Let us deal with the critical case $p=p_{\Fuj}(Q)$.
Combining \eqref{a priori est hom sol} and \eqref{a priori est Duh sol log}, we obtain
\begin{align*}
\| \Phi[u] \|_{X_{Q,T}}
& \leq  C_0\, \varepsilon \|(1+ |\cdot|_{\mathbf{H}_n}^2)^{\frac{\kappa}{2}}u_0\|_{ L^\infty(\mathbf{H}_n)}+   C_2 \log(e+T) \big\|\big(1+t+|\eta|_{\mathbf{H}_n}^2\big)^{\frac{Q}{2}+1}|u(s,\eta)|^p\big\|_{L^\infty(0,T; L^\infty(\mathbf{H}_n))}\\
& = C_0\, \varepsilon \|(1+ |\cdot|_{\mathbf{H}_n}^2)^{\frac{\kappa}{2}}u_0\|_{ L^\infty(\mathbf{H}_n)}+  C_2 \log(e+T)\big\|\big(1+t+|\eta|_{\mathbf{H}_n}^2\big)^{\frac{Q p}{2}}|u(s,\eta)|^p\big\|_{L^\infty(0,T; L^\infty(\mathbf{H}_n))} \\
& = C_0 R_0\, \varepsilon + C_2 \log(e+T) \|u\|_{\mathrm{X}_{Q,T}}^p,
\end{align*}  and, similarly,
\begin{align*}
\| \Phi[u] -\Phi[v] \|_{X_{Q,T}} & \leq p \,   C_2 \log(e+T)  \| u-v\|_{X_{Q,T}} \Big(\| u\|_{X_{Q,T}}^{p-1}+\| v\|_{X_{Q,T}}^{p-1}\Big).
\end{align*} If we assume that $T\leq \exp\big(\widetilde{C}_Q\varepsilon^{-(p-1)}\big)$, where $\widetilde{C}_Q\doteq \min \left\{2C_2R^{1-p},4p\, C_2R^{1-p}\right\}$, then, $\Phi$ satisfies \eqref{Phi is a contraction}. Hence, we have a local solution to \eqref{semilinear heat Heisenberg} in the critical case at least until the time $\exp\big(\widetilde{C}_Q \varepsilon^{-(p-1)}\big)$. Therefore, we showed the lower bound estimate of the lifespan in \eqref{lower bound lifespan} for the critical case as well.
So, the proof is complete.
\end{proof}

\section{Concluding remarks}

Let us summarize what we proved in the main results of this paper. Combining Theorem \ref{Thm GESD} and Proposition \ref{Prop test function method}, we get that $p_{\Fuj}(Q)$ is the critical exponent for the semlinear Cauchy problem \eqref{semilinear heat Heisenberg}. So, in the Heisenberg group the critical exponent for the semilinear heat equation with power nonlinearity is exactly the exponent which is the analogous one of Fujita exponent, obtained by replacing the dimension of $\mathbb{R}^n$  by the homogeneous dimension $Q=2n+2$ of $\mathbf{H}_n$.  Furthermore, we proved the sharp lifespan estimate for local solutions both in the subcritical and in the critical case, namely,
\begin{align*}
T_\varepsilon \simeq \begin{cases} C\varepsilon^{-(\frac{1}{p-1}-\frac{Q}{2})^{-1}} & \mbox{if} \ \ 1<p<p_{\Fuj}(Q), \\ \exp\big(C \varepsilon ^{-(p-1)}\big) & \mbox{if} \ \ p=p_{\Fuj}(Q). \end{cases}
\end{align*} Note that also for the lifespan estimate the situation is completely analogous to the Euclidean case.

\section*{Acknowledgments}

V. Georgiev is supported in part by
GNAMPA - Gruppo Nazionale per l'Analisi Matematica,
la Probabilit\`a e le loro Applicazioni,
by Institute of Mathematics and Informatics,
Bulgarian Academy of Sciences and Top Global University Project, Waseda University,  by the University of Pisa, Project PRA 2018 49.
A. Palmieri is supported by the University of Pisa, Project PRA 2018 49.



\addcontentsline{toc}{chapter}{Bibliography}


\end{document}